\documentclass[11pt]{amsart}
\usepackage{amsmath, amsthm, amssymb}
\usepackage{amsfonts}
\usepackage{fullpage}
\usepackage{color}
\usepackage{hyperref}
\usepackage{enumitem}
\usepackage{bm}
\usepackage{verbatim}

\usepackage[dvipsnames]{xcolor}
\usepackage{graphicx}
\DeclareGraphicsExtensions{.pstex,.eps}

\usepackage{xcolor}

\renewcommand{\ell}{{l}}  

\newcommand{\R}{{\mathbb{R}}}
\newcommand{\Z}{{\mathbb Z}}
\newcommand{\C}{{\mathbb C}}

\renewcommand{\P}{{\mathbb P}}
\renewcommand{\S}{{\mathbb S}}

\newcommand{\GG}{{\mathcal G}}

\newcommand{\NN}{{\mathcal N}}

\newcommand{\QQ}{{\mathcal Q}}

\newcommand{\tOm}{{\tilde \Om}}
\newcommand{\tS}{{\tilde S}}

\renewcommand{\Re}{\mathop{\rm Re}\nolimits}
\renewcommand{\Im}{\mathop{\rm Im}\nolimits}

\newcommand{\lf}{L^\infty}

\renewcommand{\div}{\operatorname{div}}

\theoremstyle{plain}

\newtheorem{thm}{Theorem}[section]
\newtheorem{prop}[thm]{Proposition}
\newtheorem{cor}[thm]{Corollary}
\newtheorem{lemma}[thm]{Lemma}

\theoremstyle{definition}

\newtheorem{rem}{Remark}[thm]

\numberwithin{equation}{section}

\newcommand{\secref}[1]{Section~\ref{#1}}

\def\squarebox#1{\hbox to #1{\hfill\vbox to #1{\vfill}}}

\newcommand{\<}{\langle}
\renewcommand{\>}{\rangle}

\renewcommand{\d}{\partial}
\newcommand{\ep}{\epsilon}
\newcommand{\lV}{\lVert}
\newcommand{\rV}{\rVert}

\def\ga{\gamma}

\def\de{\delta}
\def\De{\Delta}
\def\ep{\epsilon}

\def\la{\lambda}
\def\La{\Lambda}

\def\om{\omega}
\def\Om{\Omega}

\def\nab{\nabla}

\def\al{\alpha}
\def\les{\lesssim}
\def\c{\cdot}

\title[Magnetoelasticity]
{Asymptotic behavior of 3-D evolutionary model of Magnetoelasticity for small data}

 \author[X. Hao]{Xiaonan Hao$^1$}
 \author[J. Huang]{Jiaxi Huang$^2$}
 \author[N. Jiang]{Ning Jiang$^{3}$}
\author[L. Zhao]{Lifeng Zhao$^{4}$}

\address{$^1$School of Mathematics and Physics, University of Science and Technology Beijing,
	\newline\indent
	Beijing
	100083, P.R. China}
\email{\href{mailto:xn\_hao@163.com}{xn\_hao@163.com}}

\address{$^2$School of Mathematics and Statistics, Beijing Institute of Technology,
		\newline\indent
				Beijing
100081, P.R. China}
\email{\href{mailto:jiaxih@bit.edu.cn}{jiaxih@bit.edu.cn}}

\address{$^3$School of Mathemtics and Statistics, Wuhan University,
	\newline\indent
	Wuhan 430072, P. R. China}
\email{\href{mailto:njiang@whu.edu.cn}{njiang@whu.edu.cn}}

\address{$^4$School of Mathemtical Sciences, University of Science and Technology of China, 
	\newline\indent
	Hefei 230026, P. R. China}
\email{\href{mailto:zhaolf@ustc.edu.cn}{zhaolf@ustc.edu.cn}}


	
\begin{document}
		
\begin{abstract}
In this article, we consider the evolutionary model for magnetoelasticity with vanishing external magnetic field and vanishing viscosity/damping, which is a nonlinear dispersive system. The global regularity and scattering of the evolutionary model for magnetoelasticity under small size of initial data is proved. Our proof relies on the idea of vector-field method due to the quasilinearity and the presence of convective term. 
A key observation is that we construct a suitable energy functional including the mass quantity, which enable us to provide a good decay estimates for Schr\"odinger flow.
In particular, we establish the asymptotic behavior in both mass and energy spaces for Schr\"odinger map, not only for gauged equation.
\end{abstract}

\date{\today}
\maketitle

		
\setcounter{tocdepth}{1}

\section{Introduction}
\subsection{Modeling of magnetoelasticity}
In the article we focus on the magnetoelastic materials, which respond mechanically to applied magnetic fields and react with a change of magnetization to mechanical stresses. These materials are so-called smart materials. Because of their remarkable response to external stimuli, they are attractive not only from the point of
view of mathematical modeling but also for applications. Magnetoelastic materials
are, among others, used in sensors to measure force or torque (cf., e.g., \cite{BS02,BS04,GRRC}) as
well as magnetic actuators (cf., e.g., \cite{SNR}) or generators for ultrasonic sound (cf., e.g., \cite{BV}).

The discovery of magnetoelasticity dates back at least to the 19th century (see \cite{Bro66}). Until the 1960s, in Brown’s monograph \cite{Bro66}, the first rigorous phenomenological theory of magnetoelasticity was built, using both Lagrangian and Eulerian coordinate systems
in the description. Tiersten also presented an essentially equivalent theory for magnetoelastic solids in \cite{Ti1,Ti2}.

Recently, in \cite{BFLS,Fo}, an energetic variational approach was taken to formulate the fully nonlinear problem
of magnetoelasticity completely in Eulerian coordinates in the current configuration. A transport equation for the deformation gradient is found to allow one to obtain the deformation gradient in the current configuration from the velocity gradient. Consequently, the major obstacle from the point view of elasticity, the invertibility of the deformation
is resolved. As for the magnetic part, the evolution of magnetization is modeled by Landau-Lifshitz-Gilbert (LLG) equation with the time derivative replaced the convective one, which is in order
to take into account that changes of the magnetization also occur due to transport by underlying
viscoelastic material. By this approach, the evolution of magnetoelasticity is modeled by the following system. Precisely, the evolutionary model of magnetoelasticity consists of the following equations of the
velocity field $u(t,x)\in \R^d$, the deformation gradient $F(t,x)\in \R^d\times \R^d$ and the magnetization $\phi(t,x) \in \S^{2}$ on the domain $(t,x)\in \R^+\times \R^d$:
\begin{equation}\label{PHLC}
	\begin{aligned}
		\left\{ \begin{aligned}
			&\d_t u+u\cdot \nab u+\nab p+\nab\cdot(2a\nab \phi\odot\nab \phi-W'(F)F^{T})-\nu \De u=\mu_0(\nab H_{ext})^T \phi\,, \\
                &\div u=0\,,\\
			&\d_t F+u\cdot\nab F-\kappa \De F=\nab u F\,,\\
			&\d_t \phi+u\cdot\nab \phi = -\ga \phi\times (2a\De \phi+\mu_0 H_{ext})-\la \phi\times [\phi\times (2a\De \phi+\mu_0 H_{ext})]\,.
		\end{aligned}\right.
	\end{aligned}
\end{equation}
In the system, the first equation of the bulk velocity $u(t,x)\in\R^d$ is the balance of
momentum in Eulerian coordinates with stress tensor $\mathcal T=-p I+\nu (\nab u+(\nab u)^T)+ W'(F)F^T-2a\nab\phi\odot \nab\phi$, where $p(t,x)\in \R$ is the hydrodynamic pressure, $(\nab\phi \odot\nab\phi)_{ij}=\d_i \phi\cdot \d_j\phi$ and $W(F)$ is the elastic energy. The matrix $F(t,x)=(F_{ij}(t,x))_{1\leq i,j\leq d}\in \R^d\times \R^d$ represents the deformation gradient of flow map $x(t,X)$ with respect to the initial position $X$, i.e. in the Lagrangian coordinates $\bar{F}_{ij}(t,X)=\frac{\d x^i(t,X)}{\d X^j}$, where $F_{ij}(t,x)$ is the entries of the $i$-th row and the $j$-th column of $F(t, x)$. In the Eulerian coordinates, the deformation gradient evolves along the third equation in \eqref{PHLC}.
The last equation of \eqref{PHLC} is the Landau-Lifshitz-Gilbert equation with the effective magnetic field $2a\De \phi + \mu_0 H_{ext}$, where
$\phi(t, x) \in \S^2$ stands for the magnetization and $H_{ext}(t,x)\in \R^d$ denotes the given external magnetic field.
The constant $\nu\geq 0$ is the viscosity coefficient of the fluid.
$\gamma>0$ is the electron gyromagnetic ratio,
$\lambda\geq 0$ is a phenomenological damping parameter, $a,\mu_0>0$ are the parameters coming from the Helmholtz free energy. The constraint $\div u=0$ is the incompressibility condition.

The system \eqref{PHLC} can be viewed as a nonlinear coupling of hydrodynamics of viscoelasticity and Landau-Lifshitz-Gilbert equation, each of which has been extensively studied in the past two decades. The $u(t,x)$ is velocity of fluid satisfying Euler (or Navier-Stokes) equations, and the $F(t,x)$ is the deformation gradient tensor with respect to the velocity field
$u(t,x)$. Here the above two equations can be seen as a (parabolic)-hyperbolic system.
When the magnetization $\phi$ of system \eqref{PHLC} is a constant unit vector, the system \eqref{PHLC} reduces to incompressible (visco)elastodynamics. Lin-Liu-Zhang \cite{LLZ-CPAM2005} and Chen-Zhang \cite{CZ-2006} studied the the global existence of incompressible elastodynamics for small data. Sideris-Thomases \cite{ST-2005,ST-2007} and Lei \cite{Lei2016} investigated the corresponding inviscid case. In a different way, Wang \cite{Wang-2017} established the global existence and the asymptotic behavior for the 2D incompressible isotropic elastodynamics for sufficiently small and smooth initial data in the Eulerian coordinates formulation.
Subsequently, Cai-Lei-Lin-Masmoudi \cite{CLLM-CPAM2019} showed  the vanishing viscosity limit for incompressible viscoelasticity.
For the more complete researches in this direction, one please refers to \cite{Lin-CPAM2012}.

The model \eqref{PHLC} is reduced to the well-known Landau-Lifshitz-Gilbert (LLG) equations of maps from $\R^d$ into $\S^2$ if we set $(u,F)\equiv (0,I)$, which is an important model known as the ferromagnetic chain system. 
Here we are interested in the special case, i.e. LLG without damping, called Landau-Lifshitz equation.
The Landau-Lifshitz equation is a type of Schr\"odinger flow to the unit sphere, which is a Hamiltonian flow with the Hamiltonian function $E(\phi)=\|\nab\phi\|_{L^2}^2$, see \cite{DW98}.
The local well-posedness theory of Schr\"odinger map flows was established by Sulem-Sulem-Bardos \cite{SSB}, Ding-Wang \cite{DW} and McGahagan \cite{M} on the domain $\R^d$, and by Chen-Wang \cite{ChenWang23} for the Neumann boundary value problem on smooth bounded domain.  The global well-posedness theory was started by Chang-Shatah-Uhlenbeck \cite{CSU} and Nahmod-Stefanov-Uhlenbeck \cite{NSU}.
For $\S^2$ target, the first global well-posedness result for Schr\"odinger flows in critical Besov spaces in $d\geq 3$ was proved by Ionescu-Kenig
\cite{IK} and Bejenaru \cite{B} independently. This was later improved to global regularity for small data in the critical Sobolev spaces in dimensions $d\geq 4$ in \cite{BeIoKe} and in dimensions $d\geq 2$ in \cite{BIKT}.
However, the question of small data global well-posedness in critical Sobolev spaces for general compact K\"ahler targets was more complicated, which was raised by Tataru in the survey report \cite{KTV}. Recently, Li \cite{Li0,Li} solved this problem using a novel bootstrap-iteration scheme to reduce the gauged equation to an approximate constant curvature system in finite times of iteration, and established the scattering for Schr\"odinger map in energy space. We can refer to \cite{KTV} for a more detailed review.

However, for the coupling \eqref{PHLC}, the research on the well-posedness is very few. In \cite{BFLS} and \cite{KKS}, the regularized transport equation version was considered, i.e. the system \eqref{PHLC} in the non-physical case $\kappa>0$. Bene\u sov\'a-Forster-Liu-Schl\"omerkempe \cite{BFLS} established the global in time weak solution by Galerkin method and a fixed point argument. Kalousek-Kortum-Schl\"omerkemper \cite{KKS} possessed global in time weak solutions, local-in-time existence of strong solutions and the weak-strong
uniqueness property.
Recently, Jiang-Liu-Luo \cite{JLL} considered the system \eqref{PHLC} in the physical case $\kappa=0$, $\nu,\ga,\la>0$ and elastic energy $W(F)=\frac{1}{2}|F|^2$. They proved the local-in-time existence of the evolutionary model for magnetoelasticity with finite initial energy by employing the nonlinear iterative approach.
They further reformulated the system near the constant equilibrium for magnetoelasticity with vanishing external magnetic field $H_{ext}$ and dissipative terms, then proved the global well-posedness for small initial data.

From the above analytical results, we know that the viscosity $\nu,\kappa>0$ and damping $\lambda>0$ play crucial roles in the well-posedness for evolutionary model of magnetoelasticity. The positivity of these coefficients bring some strong damping effect, which dominates the dispersion and wave features of the system.  There are some special but still important cases where the viscosity and damping effects are very weak. So physically these effects are simply neglected, i.e. the corresponding coefficients are set as $0$, i.e.  $\nu=0,\ \kappa=0,\  \lambda=0$. However, in this setting, the system \eqref{PHLC} has been even changed the types of equations. More specifically, it would be reduced to a nonlinear dispersive equations, 
where we need more ideas and observations compared with \cite{JLL}.

\subsection{The main results}
Our objective in this paper is to establish the global in time well-posedness and scattering for solutions to the evolutionary model of magnetoelasticity without viscosity, damping and external magnetic field for small initial data. A key observation is that
constructing an effective energy functional is a delicate matter, which a-priori requires a good frame on the vector bundle. This is done in the next section, where we fix the gauge and write
the equation as a Schr\"odinger evolution in a good gauge.

Since we are interested in the model of magnetoelasticity without viscosity, damping and external magnetic field, thus we set
\begin{equation*}
    \nu=0\,,\quad \lambda=0\,,\quad H_{ext}=0\,.
\end{equation*}
Because the constants are irrelevant for mathematical analysis, we can also set
\[a=1/2\,,\quad \gamma=1\,.\]
Here we take the elastic energy $W(F) = \frac{1}{2}|F|^2 $.

To analyze the long time behaviors of \eqref{PHLC}, we need more constraints with respect to $F$. Given a family of deformation $x(t,X)$, the deformation gradient is defined as $F=\frac{\d x(t,X)}{\d X}\big|_{X=X(t,x)}$. At initial time $t=0$, it is easy to see that $F(0,x)=I$, hence $\nab\cdot F^T(0,x)=\sum_j\d_j F_{jk}(0,x)=0$. Moreover, it can be shown that $\nab\cdot F^T$ satisfies a transport equation, see \cite[Page 237]{LW}. Hence,  the following relation will hold for all the time, and can be imposed on the system \eqref{PHLC}
\begin{align*}
    (\nab\cdot F^T)_k=\sum_j\d_j F_{jk}=0\,, \qquad k=1,2,3\,.
\end{align*}
Since the deformation gradient is expressed as $\bar{F}_{ij}(t,X)=\frac{\d x^i(t,X)}{\d X^j}$ in the Lagrangian coordiantes, then we obtain $\d_{X^k}\bar{F}_{ij}(t,X)=\d_{X^j}\bar{F}_{ik}(t,X)$, which is also read as the constraint as follows, in the Eulerian coordinates,
\begin{align*}
    \sum_{m}F_{mj}\d_m F_{ik}=\sum_m F_{mk}\d_m F_{ij}\,,\qquad  i,j,k=1,2,3\,.
\end{align*}
Lastly, the constraint $|\phi|=1$ comes from the geometric constraint $\phi\in \S^2$.

Therefore, the system that we are going to consider in this article is as follows
\begin{equation}        \label{ori_sys}
	\left\{
	\begin{aligned}
		&\d_t u+u\cdot \nab u+\nab p=\nab\cdot(FF^T-\nab \phi\odot\nab \phi)\,,\\
		&\d_t F+u\cdot\nab F=\nab u F\,,\\
		&\d_t \phi+u\cdot \nab \phi= -\phi\times \De \phi\,,\\
		&(u,F,\phi)\big|_{t=0}=(u_0,F_0,\phi_0)\,,
	\end{aligned}
	\right.
\end{equation}
on $\R^+\times \R^d$ with the constraints
\begin{align} \label{constraints1-re}
	&\div u=0\,,\quad \nab\cdot F^T=0\,,\quad
	\sum_m F_{mj}\d_m F_{ik}=\sum_m F_{mk}\d_m F_{ij}\,,\quad |\phi|=1\,.
\end{align}
The above evolutionary model of magnetoelasticity admits two conservation laws for solutions. One of the conservaltion laws of \eqref{ori_sys} is conservation of energy defined by
\begin{equation}    \label{EConsLaw}
	E(t)=\frac{1}{2}\int_{\R^d} |u|^2+\sum_{i,j}F^2_{ij}+\sum_k |\d_k \phi|^2\  dx\,.
\end{equation}
If the Schr\"odinger map satisfies $\|\phi_0-Q\|_{L^2_x}<\infty$ at initial time for some $Q\in \S^2$, the Schr\"odinger map flow has mass as another conserved quantity:
\begin{equation}   \label{MConsLaw}
	M(t)=\frac{1}{2}\int_{\R^d} |\phi-Q|^2\ dx\,.
\end{equation}

The local well-posedness of \eqref{ori_sys}-\eqref{constraints1-re} for large initial data is proved in \cite{HJJ} by approximation of perturbed parabolic system and parallel transport method.

\begin{thm} [Local well-posedness]      \label{main-Thm}
Let $d=2,3$ be the dimensions, $Q\in \S^2$ be a fixed unit vector, and let initial data $(u_0,F_0,\phi_0)\in H^3\times H^3\times H^{4}_Q$. Then the Cauchy problem \eqref{ori_sys} with constraints \eqref{constraints1-re} admits a unique local solution $(u,F,\phi)$ on $[0,T]$ satisfying
	\begin{equation*}
		\lV u\rV_{H^3}+\lV F\rV_{H^3}+\lV \phi\rV_{H^{4}_Q}\leq C_3(\lV u_0\rV_{H^3},\lV F_0\rV_{H^3},\lV \phi_0\rV_{H^{4}_Q})\,,
	\end{equation*}
	where the time interval $T$ depends on initial data $\lV u_0\rV_{H^3},\lV F_0\rV_{H^3}$ and $\lV \phi_0\rV_{H^4_Q}$.
	Moreover, if the initial data $(u_0,F_0,\phi_0)\in H^k\times H^k \times H^{k+1}_Q$ for any integer $k\geq 3$, then we have
\begin{equation*}
\lV u\rV_{H^k}+\lV F\rV_{H^k}+\lV \phi\rV_{H^{k+1}_Q}\leq C_k(k,\lV u_0\rV_{H^k},\lV F_0\rV_{H^k},\lV \phi_0\rV_{H^{k+1}_Q})\,.
\end{equation*}
\end{thm}

Our main theorem is stated as follows.
Here the precise definitions of the function spaces will be given in Section \ref{sec-FS}.

\begin{thm}[Global regularity and scattering for small data in 3-D]           \label{Ori_thm0}
Let integer $N\geq 9$ and $0<\de\leq 1/8$ be two given constants, and $Q\in \S^2$ be a fixed unit vector.  the initial data $(u_0,F_0-I,\phi_0)\in H^N_\Lambda$. Then there exists $\ep_0>0$ sufficiently small depending on $N$ and $\de$ such that, for all initial data $(u_0,F_0,\phi_0)$ satisfying the constriants \eqref{constraints1-re} with
\begin{equation}  \label{Main-ini0}
\sum_{|a|\leq N}\big(\|\Lambda^a (u_0,F_0-I)\|_{L^2_x}+\|  \Lambda^a (\phi_0-Q)\|_{L^2_x} +\|\nab \Lambda^a \phi_0\|_{L^2_x}\big)\leq \ep_0\,,
\end{equation}
the evolutionary model \eqref{ori_sys} of magnetoelasticity with constraints \eqref{constraints1-re} is global in time well-posed in three dimensions. Moreover, for all $t\in[0,+\infty)$ the solution $(u,F-I,\phi)\in H^{N}_Z$ has the bounds
	\begin{equation}  \label{energybd0}
		E^{1/2}_N(u,F-I,\phi;t)\les \ep_0\<t\>^\de\,,\quad E^{1/2}_{N-2}(u,F-I,\phi;t)\les \ep_0\,.
	\end{equation}
In addition, for any multiindex $|a|\leq N-4$, there exists matrix $\GG^{(a)}_\infty: \R^3\rightarrow \C^{3\times 3}$ in $L^2_x$ such that
 \begin{align} \label{Scatter-G}
 &\lim_{t\rightarrow \infty}\|Z^a (F-I)-\Re \big(e^{-it|\nab|}\mathcal G^{(a)}_\infty\big)\|_{L^2}+\| Z^a u+\sum_{j=1}^3\Im\big(e^{-it|\nab|}|\nab|^{-1}\d_j\mathcal G^{(a)}_{\infty,\cdot j}\big)\|_{L^2_x}=0\,,
\end{align}
and, in the mass and energy spaces, there exist two complex vectors $\Phi_{1,a},\ \Phi_{2,a}: \R^3\rightarrow\C^3$ in $H^1_x$ such that
\begin{equation}  \label{scattering0}
	\lim_{t\rightarrow\infty} \|Z^a(\phi-Q)-\Re(e^{-it\De}\Phi_{1,a})-\Im(e^{-it\De}\Phi_{2,a})\|_{H^1_x}=0\,.
\end{equation}
\end{thm}

\begin{rem}
    We will be working in a localized energy space defined by commutative vector fields $Z$ and the time independent analogue $\Lambda$, since the evolutionary model of magnetoelasiticiy is a nonlinear coupling of quasilinear elastic equations and quadratic Schr\"odinger equation. The definitions of $Z,\ \Lambda,\ H^N_\Lambda,\ H^N_Z$ and $E_N$ will be given in Section \ref{sec-FS}. In particular, the localized energy space includes the mass $Z^a(\phi-Q)\in L^2_x$, which plays an important role in decay estimates and asymptotic analysis.
\end{rem}

\begin{rem}
    i) The asymptotic behaviors in \eqref{Scatter-G} are new even for elastic equations under vector-field method. 
    ii) In addition to energy space $\dot H^1$, the asymptotic behavior \eqref{scattering0} also holds in mass space $L^2$, which is slightly different from the results in \cite[(1.5)]{Li} by Ze Li. There the first asymptotic behavior in energy space for Schr\"odinger map, not only for gauged equation, was provided.
\end{rem}

\subsection{An overview of the paper}
Our first objective in this article will be to provide a suitable formulation of the Schr\"odinger flow, interpreted as a nonlinear Schr\"odinger equation where the nonlinearities admit good structure.
The formulation \eqref{ori_sys} has a key gauge freedom for the orthonormal frame on
the tangent bundle, where a suitable gauge should be chosen to obtain a suitable Schr\"odinger equation. This is the key step in order to analyze the long-time behaviors of \eqref{ori_sys}.

This reformulation of \eqref{ori_sys} is stated in \secref{sec-2}, where we mainly rewrite the Schr\"odinger flow as a nonlinear Schr\"odinger equation for well chosen variables. These independent variables, denoted by $\psi_m$, represents the projection of the tangent vector fields on $t$, in complex notation. In addition to the variables,
we will use several dependent variables, as follows:
\begin{itemize}
    \item The magnetic potential $A_m$ for $m=1,\cdots,d$, associated to the natural connection on the tangent bundle $T\S^2$,
    \item The corresponding temporal component $A_{d+1}$.
\end{itemize}
These additional variables will be viewed as uniquely determined by our
variables $\psi$, provided that a suitable gauge choice was made. This is done by choosing Coulomb gauge on $T\S^2$, i.e. $\div A=0$.
In the next section we describe the derivations of reformulation, so that by the
end we obtain a nonlinear Schr\"odinger equation for $\psi$, with $A$ and $A_{d+1}$ determined by $\psi$ in an elliptic fashion, together with suitable compatibility conditions (constraints).

Setting the stage to solve these equations, in Section \ref{sec-FS} we describe the function spaces
for the unknowns $u,\ G$ and $\psi$. First, we introduce the commutative vector fields corresponding to the coupled system, which are used to define the Klainerman's generalized energy and $L^2$ weighted norms. Second, we provide a lemma in order to yield the initial data for reformulated system, and then we state our second main theorem and bootstrap proposition.

The main strategy to establish global regularity and scattering in Theorem \ref{Ori_thm0} relies on an interplay
between the control of high order energies and decay estimates, which is based on the vector-fields
method and weighted energy method.
In Section \ref{sec-dec}, we review several weighted and decay estimates, then the decays of velocity $u$, displacement gradient $G=F-I$ and differential fields $\psi$ are obtained. In particular, as a corollary we also establish the decay estimates for connection $A,\ A_{d+1}$ and the frame $(v,\ w)$ near $v_{\infty}=\lim_{|x|\rightarrow\infty}v(x)$ and $w_{\infty}=\lim_{|x|\rightarrow\infty}w(x)$ respectively.
With the above decays at hand, in Section \ref{sec-EnEs}, we establish the energy estimates for higher-order and low-order energy norms, respectively. 
In Section \ref{sec-L2}, we establish the $L^2$ weighted estimates,   
whose proof relies on the Klainerman-Sideris type estimate and the special structure of nonlinearities.
The analysis is completed in Section \ref{sec-proof}, where we use the energy and $L^2$ weighted estimates in order to (i) prove the bootstrap proposition, then (ii) to establish global existence and scattering of the reformulated system. Finally, we construct solutions for the Schr\"odinger flow and the scattering \eqref{scattering0} in mass and energy spaces.

The analysis here broadly follows the vector-field method
introduced by Klainerman in \cite{Kla85}, but a number of improvements are needed which allow us to take
better advantage of the structure of the magnetoelastic equations.
In order to establish our main result, some key observations and novelties are emphasized as follows.

\emph{1. Energy norms and decay estimates.}
The choice of localized energy norms plays an important role in our proof. As mentioned in pioneer work \cite{Kla85}, the vector fields should be used in order to gain decay estimates due to the quasilinearity of systerm \eqref{ori_sys}. However, this does not enable us to obtain a good decay estimates for the Schr\"odinger flow. Fortunately, noticing that the system not only has the energy conservation law \eqref{EConsLaw}, but also has the mass conservation law \eqref{MConsLaw}, which suggests us defining our energy norms by 
\begin{align*}
	E_j(t)=\sum_{|a|\leq j}\Big(\|Z^a u\|_{L^2}^2+\|Z^a (F-I)\|_{L^2}^2+\|\nab Z^a\phi\|_{L^2}^2+\|Z^a(\phi-Q)\|_{L^2}^2\Big)\,,
\end{align*}
where the vector fields $Z$ are given in \secref{sec-FS}. 
We remark that the last term is the key for the decay estimates of Schr\"odinger flow and scattering \eqref{scattering0}.

\emph{2. Cancellations and geometric structure.}
The lack of viscosity and damping in the system \eqref{ori_sys} forces us to investigate the cancellations and geometric structure. We must point out that the model \eqref{ori_sys} has two cancellations
\begin{align*}
    \<u,\nab\cdot (\nab\phi\odot\nab\phi) \>+\<\nab\phi,\nab(u\cdot\nab\phi)\>=0\,,\\
    \<u,\nab\cdot (FF^T)\>+\<F,\nab u F\>=0\,,
\end{align*}
and a compatibility condition for Schr\"odinger map
\begin{equation*}
    (\d_j+iA_j)\psi_k=(\d_k+iA_k)\psi_j\,,
\end{equation*}
which play an essential role in the derivation of the a-priori estimates to the system \eqref{ori_sys}. 
Since we reformulate the Schr\"odinger flow in Coulomb gauge, we can also rephrase the above first cancellation as 
\begin{equation*}
	-\<u,\d_j \Re(\psi\overline{\psi_j})\>-\<\psi_j, \Re(\d_j u\cdot\overline{\psi})\>=0\,.
\end{equation*}
On one hand, the above three conditions enable us get through the derivative loss, and are used in Section \ref{sec-EnEs} to prove higher-order energy estimates. On the other hand, we also must utilize them to gain more decays, especially dealing with the convective terms.

\emph{3. Ths scattering of Schr\"odinger flow.} The energy norms including the mass $\|Z^a(\phi-Q)\|_{L^2}$ motivates us to gain a better scattering for Schr\"odinger map. Precisely, we can prove that the nonlinearities in $H^{-1}$ of Schr\"odinger equation for $\psi$ are decays with the rate $\<t\>^{-4/3+\de}$ at least, then the differential field $\psi$ scatters in the function space $H^{-1}$. Return to the map $\phi$, we can establish the scattering of Schr\"odinger flow in mass and energy spaces correspondingly.

\medskip
\section{Reformulation and function spaces}\label{sec-2}
In this section, first we reformulate the system \eqref{ori_sys} by introducing our main independent variable $G$ and $\psi$. Second, we introduce the commutative vector fields corresponding to the symmetries, and then define the energy norms and weighted $L^2$ generalized energies. In the end, we state our main result and the main bootstrap proposition for system \eqref{sys-2}.

\subsection{Reformulation of \eqref{ori_sys}}
Since the deformation gradient $F$ is near the equilibrium $I$, we denote the \emph{displacement gradient} as
\begin{equation*}
	G:=F-I\,.
\end{equation*}
Then the equation of $F$ in \eqref{ori_sys} is written as
\begin{equation*}  
	\d_t G-\nab u =-u\cdot \nab G+\nab u G\,,
\end{equation*}
with the associated constraints expressed as 
\begin{align*}   
	\d_j G_{ik}-\d_k G_{ij}=G_{mk}\d_m G_{ij}-G_{mj} \d_m G_{ik},\quad  \nab\cdot G^T=0.
\end{align*}
By the above notations and $\nab\cdot G^T=0$, the quadratic term $\nab\cdot (FF^{T})$ is rewritten as
\begin{align*}
	\nab\cdot(FF^T)&=\nab\cdot[(I+G)(I+G)^T]=\nab\cdot G+\nab\cdot G^T+\nab\cdot(GG^T)=\nab\cdot G+\nab\cdot(GG^T).
\end{align*}
Then the equation of $u$ in \eqref{ori_sys} is given by
\begin{equation*}   
\d_t u-\nab\cdot G+\nab p=-u\cdot \nab u+\nab\cdot(GG^T)-\div  (\nab \phi\odot \nab \phi).
\end{equation*}

Next, we turn our attention to the reformulation of Schr\"odinger flow. 
Assume that the smooth function $\phi:[0,T]\times \R^n\rightarrow \S^2$ satisfies the Schr\"odinger map equation in \eqref{ori_sys}.
Let $v,w\in T_{\phi}\S^2$ be tangent vector fields satisfying $\phi\times v=w, w\times \phi=v$.
Then we define the differentiated variables $\psi_m$ for $m=1,\cdots,d,d+1$, connection coefficients $A_m$ and $A_{d+1}$,
\begin{equation}   \label{diff-field}
\psi_m=\d_m \phi\cdot v+i\d_m \phi\cdot w,\qquad A_m=\d_m v\cdot w, \qquad A_{d+1}=(\d_t v)\cdot w.
\end{equation}
These allow us to express $\d_m \phi,\ \d_m v$ and $\d_m w$ in the frame $(\phi,v,w)$ as
\begin{equation}    \label{phi-v-w}
\left\{ \begin{aligned}
&\d_m \phi=v\Re \psi_m+w\Im \psi_m,\\
&\d_m v=-\phi\Re \psi_m+w A_m,\\
&\d_m w=-\phi\Im \psi_m- v A_m.
\end{aligned}
\right.
\end{equation}
Denote the covariant derivative as
\begin{equation*}
D_l=\d_l+iA_l.
\end{equation*}
We then obtain the compatibility and the curvature of connection
\begin{equation}     \label{com-Dpsi}
D_l\psi_m=D_m\psi_l,\qquad D_lD_m-D_mD_l=i(\d_l A_m-\d_m A_l)=i\Im (\psi_l\Bar{\psi}_m).
\end{equation}

Now we derive the Schr\"odinger equations for the variables $\psi_m$.
By \eqref{ori_sys}, $\phi\times v=w$ and $w\times \phi=v$, we have
\begin{align} \label{psi_d+1}
\psi_{d+1}=-u\c \psi-i\sum_{l=1}^d D_l\psi_l.
\end{align}
Applying $D_m$ to \eqref{psi_d+1}, by the relations \eqref{com-Dpsi} we obtain
\begin{align*}
D_{d+1}\psi_m=-\d_m u\c \psi-u\c D\psi_m-\sum_{l=1}^d\Im(\psi_l\bar{\psi}_m)\psi_l-i\sum_{l=1}^d D_lD_l\psi_m,
\end{align*}
which is equivalent to
\begin{equation}\label{Sch-eq}
    \begin{aligned}
i\big(\d_t+(u-2A)\c \nab\big)\psi_m-\De \psi_m=&\ (A_{d+1}+(u-A)\c A+i\nab\c A)\psi_m\\
&\ -i\d_m u\c \psi-\sum_{l=1}^d i\Im(\psi_l\overline{\psi_m})\psi_l\,.
\end{aligned}
\end{equation}

Consider the system of equations which consists of \eqref{Sch-eq} and \eqref{com-Dpsi}. The solution $\psi_m$ for the above system can't be uniquely determined as it depends on the choice of the orthonormal frame $(v,w)$. Precisely, it is invariant with respect to the gauge transformation
$\psi_m\rightarrow e^{i\theta}\psi_m,\ A_m\rightarrow A_m+\d_m \theta$.
In order to obtain a well-posed system one needs to make a choice which uniquely determines the gauge. Here we choose to use the Coulomb gauge
$\div A=0$,
where the existence of Coulomb gauge was proved in \cite[Proposition 2.3]{BeIoKe}.
The relation \eqref{com-Dpsi}, combined with the Coulomb gauge, leads to
\begin{equation}  \label{A_m-Eq}
\De A_m=\sum_{l=1}^d\d_l \Im (\psi_l\bar{\psi}_m)\,,\qquad m=1,\cdots, d\,.
\end{equation}
Similarly, by \eqref{com-Dpsi} and \eqref{psi_d+1} we also have the compatibility condition
\begin{align*}
\d_t A_m-\d_m A_{d+1}=\Im (\psi_{d+1}\bar{\psi}_m)=\Im [(-u\c \psi-\sum_{l=1}^d iD_l\psi_l)\bar{\psi}_m].
\end{align*}
This combined with the Coulomb gauge $\div A=0$ implies
\begin{equation}\label{A_d+1-Eq}
\begin{aligned}
\De  A_{d+1}
&=\sum_{m,l=1}^d\d_m\d_l\Re (\psi_l\bar{\psi}_m)-\frac{1}{2}\De|\psi|^2+\sum_{m=1}^d\d_m\Im (u\c \psi\bar{\psi}_m)\,.
\end{aligned}
\end{equation}

In conclusion, under the Coulomb gauge $\div A=0$, by \eqref{Sch-eq}, \eqref{A_m-Eq} and \eqref{A_d+1-Eq} we obtain the system for velocity $u$ and differentiated fields $\psi_m$
\begin{equation}\label{sys-2}
\left\{ \begin{aligned}
&\d_t u-\nab\cdot G+\nab p=-u\cdot \nab u+\nab\cdot(GG^T)-\sum_{j=1}^3\d_j\Re(\psi \bar\psi_j)\,,\\
&\d_t G-\nab u =-u\cdot \nab G+\nab u G\,,\\
&\begin{aligned}
i(\d_t+(u-2A)\c \nab)\psi_m+\De \psi_m=&(A_{d+1}+(u-A)\c A)\psi_m\\
&-i\d_m u\c \psi-i\sum_{l=1}^d\Im(\psi_l\bar{\psi}_m)\psi_l\,,
\end{aligned}\\
&(u,G,\psi)\big|_{t=0}=(u_0,G_0,\psi_0)\,,
\end{aligned}\right.
\end{equation}
where connection coefficients $A_m$ and $A_{d+1}$ are determined at fixed time in an elliptic fashion via the following equations
\begin{equation}   \label{A,Ad+1}
    \left\{\begin{aligned}
    &\De A_m=\sum_{l=1}^d\d_l \Im (\psi_l\bar{\psi}_m)\,,\\
    &\De  A_{d+1}=\sum_{m,l=1}^d\d_m\d_l\Re (\psi_l\bar{\psi}_m)-\frac{1}{2}\De|\psi|^2+\sum_{m=1}^d\d_m\Im (u\c \psi\bar{\psi}_m)\,.
\end{aligned}   \right.
\end{equation}
and $u,\ G$ satisfy the constraints
\begin{align}   \label{constraints-re}
    \div u=0,\quad \nab\cdot G^T=0,\quad
	\d_j G_{ik}-\d_k G_{ij}=\sum_m (G_{mk}\d_m G_{ij}-G_{mj} \d_m G_{ik}).
\end{align}
We can assume that the following
conditions hold at infinity in an averaged sense:
$A_m(\infty) = 0,\ A_{d+1}(\infty)=0$.
These are needed to insure the unique solvability of the above elliptic equations in a suitable class of functions. 
Therefore, to obtain our global solution for system \eqref{ori_sys}, it suffices to study the differentiated system \eqref{sys-2}-\eqref{constraints-re}.

\subsection{Function spaces and bootstrap proposition}\label{sec-FS}
Since the system admits rotational and scaling invariance, then we can introduce the associated commutative vector fields, which has been considered in \cite{CLLM-CPAM2019,CW}.
From rotational transformation, we obtain the perturbed angular momentum operators $\tilde\Om$, which acting on the unknowns $u,F,\phi$ and the associated variables in \eqref{sys-2}: $p$, $G$, $\psi$, $A$ and $A_{d+1}$ are given by
\begin{align*}
\tOm_i (u,F,\phi,p,G,\psi,A,A_{d+1})
:=&\ \big(\Om_i u +M_i \cdot u,\ \Om_i F +[M_i ,F],\ \Om_i\phi,\ \Om_i p,\ \Om_i G +[M_i ,G],\\
\ &\quad  \ 
\Om_i \psi+M_i \cdot\psi,\ \Om_i A+M_i\cdot A,\ \Om_i A_{d+1}\big)\,,
\end{align*}
where $M_j$ are the generators of rotations given by
\begin{equation*}
	M_1=\left(\begin{array}{ccc}
		0&0&0\\
		0&0&1\\
		0&-1&0
	\end{array}\right),\ \
	M_2=\left(\begin{array}{ccc}
		0&0&-1\\
		0&0&0\\
		1&0&0
	\end{array}\right),\ \
	M_3=\left(\begin{array}{ccc}
		0&1&0\\
		-1&0&0\\
		0&0&0
	\end{array}\right),
\end{equation*}
$\Om = (\Om_1, \Om_2, \Om_3)$ is the rotation vector feld $\Om = x \times \nab $, and $[M_i,G]=M_i G-GM_i$ denotes the standard Lie bracket product.
In view of the scaling invariance,
the perturbed scaling operators $\tilde S$ acting on the unknowns are given by
\begin{align*}
	\tS (u,F,\phi,p,G,\psi,A_m,A_{d+1})
	:=&\ \big((S+1)u,\ (S+1)F,\ S\phi,\ (S+2)p,\ (S+1)G,\\
       &\quad  (S+1)\psi,\ (S+1)A_m,\ (S+2)A_{d+1}\big)\,,
\end{align*}
where $S$ is the standard scaling vector field $S=2t\d_t+x\cdot\nab$.

For the convenience of use, we will occasionally utilize the notations $\tilde \Om_i,\ \tilde S$ when the above perturbed angular momentum/ scaling operators acting on the indicated component of $(u,F,\phi,p,G,\psi,$ $A,A_{d+1})$. For instance, we will write $\tilde \Om_i u=\Om_i u +M_i \cdot u,\ \tilde \Om_i F=\Om_i F +[M_i ,F],\ \tilde \Om_i\phi=\Om_i\phi$, $\tilde S u=(S+1)u, \ \tS F=(S+1)F,\ \tS \phi=S\phi$ for the unknowns.

Let
\begin{equation*}
	Z=(Z_1,\cdots,Z_8)= \big\{\d_t,\d_1,\d_2,\d_3,\tOm_1,\tOm_2,\tOm_3,\tilde S\big\}\,.
\end{equation*}
For any multi-index $a=(a_1,\cdots,a_8)\in\Z_+^8$, we denote $Z^{a}=Z_1^{a_1}\cdots Z_8^{a_8}$.
We define the Klainerman's generalized energy of system \eqref{sys-2} as
\begin{equation*}   
	E_{j}(t)=\frac{1}{2}\sum_{|a|\leq j}\int |Z^a u|^2+|Z^a G|^2+\big||\nab|^{-1}Z^a\psi\big|^2+| Z^a\psi|^2\ dx\,,
\end{equation*}
We also need the weighted $L^2$ generalized energies of elastic equations and Schr\"odinger flow
\begin{align}    \label{XX}
	X_{j}(t)=&\ \sum_{|a|\leq j-1}\big(\|\<t-r\>\nab Z^a u\|_{L^2}^2+\|\<t-r\>\nab Z^a G\|_{L^2}^2\big)\,,\\ \label{YY}
    Y_{j}(t)=&\ \sum_{|a|\leq j-1}\|x\cdot\nab |\nab|^{-1}\Psi^{(a)}\|\,,
\end{align}
where $\Psi^{(a)}$ is called \emph{profile} of $Z^a\psi$ given by
\begin{align*} 
    \Psi^{(a)}=e^{it\De} Z^a\psi\,.
\end{align*}

In order to characterize the initial data, we introduce the time independent
analogue of $Z$. The only difference will be in the scaling operator. Set
\begin{equation*}
	\Lambda=(\Lambda_1,\cdots,\Lambda_7)=(\d_1,\d_2,\d_3,\tilde \Om_1,\tilde \Om_2,\tilde \Om_3,\tilde S_0)\,,\quad \tilde S_0=\tilde S-t\d_t\,.
\end{equation*}
Then the commutator of any two $\Lambda$'s is again a $\Lambda$. Define
\begin{equation*}   
	H^m_\Lambda=\Big\{ (u_0,G_0,\psi_0):\sum_{|a|\leq m}\big(\|\Lambda^a (u_0,G_0)\|_{L^2}+\||\nab|^{-1} \Lambda^a \psi_0\|_{L^2} +\| \Lambda^a \psi_0\|_{L^2}\big) <\infty \Big\}\,.
\end{equation*}
We shall solve the evolutionary system \eqref{sys-2} in the space
\begin{align*}
	H^m_Z(T)=\Big\{ (u,G,\psi): Z^a u,Z^a G, |\nab|^{-1}Z^a\psi, Z^a\psi\in L^\infty([0,T]; L^2) ,\ \text{ for any}\ |a|\leq m      \Big\}\,.
\end{align*}
For the sake of convenience, for any integer $j\in \Z^+$ and a function space $U$, we denote
\begin{align}   \label{Zj-U}
	\|\Lambda^j f\|_{U}=\sum_{|a|\leq j}\|\Lambda^a f\|_{U}\,,\qquad \|Z^j f\|_{U}=\sum_{|a|\leq j}\|Z^a f\|_{U}\,.
\end{align}

Applying the vector fields $Z's$ to \eqref{sys-2}, we can derive that
\begin{equation}     \label{sys-VF}
	\left\{
	\begin{aligned}
		&\d_t Z^a u-\nab\cdot Z^a G+\nab Z^a p=f_a\,,\\
		&\d_t Z^a G-\nab Z^a u = g_a\,,\\
		&i\d_t Z^a \psi-\De Z^a\psi=h_a\,,
	\end{aligned}
	\right.
\end{equation}
where $A$ and $A_{d+1}$ satisfy
\begin{align}    \label{Ell-A}
\De Z^a A&=\sum_{b+c=a} \sum_l C_a^b \d_l \Im (Z^b\psi_l \overline{Z^c\psi})\,,\\ \label{Ell-Ad1}
\De Z^a A_{d+1}&=\sum_{b+c=a} C_a^b \big(\sum_{l,m} \Re\d_m\d_l\Re (Z^b\psi_l\overline{Z^c\psi_m})-\frac{1}{2}\De (Z^b\psi \cdot\overline{Z^c \psi})\big)\\\nonumber
&\quad +\sum_{b+c+e=a}\sum_m C_a^{b,c}\Im\d_m(Z^b u\cdot Z^c\psi Z^e \bar{\psi}_m)\,,
\end{align}
with constraints
\begin{align}  \label{curl-vf}
	&\div Z^a u=0\,,\quad \nab\cdot Z^a G^T=0\,, \quad  
	\d_j Z^a G_{ik}-\d_k Z^a G_{ij}=\mathcal N_{a,jik}\,,\\ \label{comm}
 & \d_j Z^a\psi_k=\d_k Z^a\psi_j+i \sum_{b+c=a} C_a^b (Z^b A_k Z^c\psi_j-Z^b A_j Z^c\psi_k)\,.
\end{align}
Here the nonlinearities are given by
\begin{align}
    &\begin{aligned}\label{fa}
    f_a:&=\sum_{b+c=a}C_a^b \big(- Z^b u\cdot \nab Z^c u +\nab\cdot (Z^b G Z^c G^T)- \d_j \Re ( Z^b\psi\cdot  \overline{Z^c\psi_j})\big)\,,
    \end{aligned}	\\   \label{ga}
	&g_a:=\sum_{b+c=a}C_a^b \big(-Z^b u\cdot \nab Z^cG+\nab Z^b u Z^c G\big)\,,
\end{align}
\begin{align}   \label{ha}
&\begin{aligned}
h_a:&= \sum_{b+c=a} C_a^b \big(-iZ^b (u-2A)\cdot \nab Z^c\psi-i\nab Z^b u\cdot Z^c\psi+Z^b A_{d+1} Z^c\psi\big)\\
&\quad +\sum_{b+c+e=a} C_a^{b,c} \big(Z^b (u-A)\cdot Z^c A Z^e\psi-i\sum_l  \Im(Z^b \psi_l \overline{Z^c\psi})Z^e\psi_l\big),
\end{aligned}\\   \label{Na}
&\mathcal N_{a,jik}:=\sum_{b+c=a}\sum_m\Big(C_a^b Z^b G_{mk}\d_m Z^c G_{ij}- Z^b G_{mj} \d_m  Z^c G_{ik}\Big)\,.
\end{align}
and the constant coefficients are defined as 
$C_a^b:=\frac{a!}{b!(a-b)!},\  C_a^{b,c}:=\frac{a!}{b!c!(a-b-c)!}$.

First, we prove quantitative estimates for the differentiated fields $\psi$ with respect to $\phi$.
\begin{lemma}[Bounds for $\psi$]\label{psi-by_phi-Lemma}
Let integer $N\geq 9$. Assume that $\phi$ is a Schr\"odinger map, the orthonormal frame $(v,w)$ satisfies the Coulomb gauge associated with map $\phi$. If the map $\phi$ has the additional property
\begin{equation}   \label{Ass}
	\sum_{|a|\leq N}\big(\|\nab \Lambda^a \phi\|_{L^2}+\| \Lambda^a(\phi-Q)\|_{L^2}\big)\leq \ep_0\,,
\end{equation}
then for the differentiated fields $\psi_m$ \eqref{diff-field} with $1\leq m\leq d$ we have the bounds
\begin{equation}\label{psi-by}
	\sum_{|a|\leq N}\big(\| \Lambda^a\psi\|_{L^2}+\| |\nab|^{-1}\Lambda^a\psi\|_{L^2}\big)\leq C_1 \ep_0\,.
\end{equation}
\end{lemma}
\begin{proof}	
\emph{i) We bound the term $\Lambda^a\psi$ in $L^2$.}
From \eqref{diff-field} and $|v|=|w|=1$, we know that $\|\psi\|_{L^2\cap L^\infty}\les \|\nab\phi\|_{H^2}\les \ep_0$. This, combined with \eqref{A,Ad+1} and \eqref{phi-v-w}, yields $\|A\|_{L^2\cap L^3}\les \|\psi\|_{L^2\cap L^3} \|\psi\|_{L^3}\les \ep_0^2$ and $\| (\nab v,\nab w)\|_{L^2}\lesssim \| \psi\|_{L^2}+\|  A\|_{L^2}\les \ep_0$. Hence, we obtain the bounds
\begin{equation}  \label{AL23}
    \|\psi\|_{L^2\cap L^\infty}+\|A\|_{L^2\cap L^3}+\| (\nab v,\nab w)\|_{L^2}\les \ep_0.
\end{equation}

Next, we bound the term $\Lambda^a \psi$ for $1\leq |a|=j\leq N$ by induction. From the above estimates we can assume that $\|\La^{j-1}\psi\|_{L^2}+y_{j-1}\les \ep_0$ with $y_k:=\|(\nab \La^{k}v,\nab \La^{k}w)\|_{L^2}$.
By \eqref{diff-field}, $|v|=|w|=1$, Sobolev embedding and \eqref{Ass}, we have
\begin{equation}\label{BdZpsi}
\begin{aligned}
\|\Lambda^j\psi\|_{L^2}
&\les \|\nab \Lambda^j\phi\|_{L^2} (1+y_{[j/2]+1}) +\|\nab \Lambda^{[j/2]}\phi\|_{L^3} y_j\les  \ep_0 (1+y_j)\,.
\end{aligned}
\end{equation}
From the equation \eqref{A,Ad+1}, $\|\La^{j-1}\psi\|\les \ep_0$ and \eqref{AL23}, we also have
\begin{align}      \label{LajA}
\|\Lambda^j A\|_{L^2}\lesssim \|\Lambda^j\psi\|_{L^2} (\|\Lambda^{j-1}\psi\|_{L^2}+\|\psi\|_{L^3})\les  \ep_0\|\Lambda^j\psi\|_{L^2}\,.
\end{align}
The formulas \eqref{phi-v-w}, combined with \eqref{LajA}, \eqref{Ass} and \eqref{AL23}, yields for $j=1$
\begin{equation}\label{bd-vw1}
\begin{aligned}
    y_1
    \les \|(\Lambda^1 \psi,\Lambda^1 A)\|_{L^2}+\|(\Lambda^1\phi,\Lambda^1v, \Lambda^1w)\|_{L^6}\|( \psi,A)\|_{L^3}
    \les  \|\Lambda^1\psi\|_{L^2}+(\ep_0+y_1)\ep_0.
\end{aligned}
\end{equation}
For $j\geq 2$, by Sobolev embeddings, \eqref{LajA}, \eqref{Ass} and $y_{j-1}\les \ep_0$, we also have 
\begin{equation}\label{bd-vw}
\begin{aligned}
y_j \les &\ \|(\phi,v, w)\|_{L^\infty}\|(\Lambda^j \psi,\Lambda^j A)\|_{L^2}
+\|(\Lambda^{[j/2]}\phi, \Lambda^{[j/2]} v, \Lambda^{[j/2]} w)\|_{L^6}\|(\Lambda^{j-1}\psi,\Lambda^{j-1}A)\|_{L^3}\\
&\ +\|(\Lambda^{j}\phi, \Lambda^{j} v, \Lambda^{j} w)\|_{L^6}\|(\Lambda^{j-2}\psi,\Lambda^{j-2}A)\|_{L^3}\\
\lesssim &\  (1+\ep_0+y_{j-1}) \| \Lambda^j\psi\|_{L^2}+(\ep_0+y_j)\|\Lambda^{j-1}\psi\|_{L^2}
\les \|\Lambda^j \psi\|_{L^2}+\ep_0^2+\ep_0 y_j\,.
\end{aligned}
\end{equation}
Then the bounds \eqref{bd-vw1} and \eqref{bd-vw} implies $y_j\les \|\Lambda^j \psi\|_{L^2}+\ep_0$, which combined with \eqref{BdZpsi} yields $\|\Lambda^j\psi\|_{L^2}\les \ep_0$.
This completes the proof of the bound \eqref{psi-by} for $\Lambda^a \psi$ with $|a|\leq N$.

\smallskip 
\emph{ii) We bound the term $|\nab|^{-1}\La^a\psi$.}
From \eqref{diff-field} and Leibniz formula it suffices to bound
\begin{align*}
\sum_{|b+c|\leq N} \big(\||\nab|^{-1}\nab \big(\Lambda^b(\phi-Q)\cdot (\Lambda^cv+i \Lambda^cw)\big)\|_{L^2}
+\||\nab|^{-1} \big(\Lambda^b(\phi-Q)\cdot (\nab \Lambda^c v+i\nab \Lambda^c w)\big)\|_{L^2}\big).
\end{align*}
By $|v|=|w|=1$, Sobolev embeddings, \eqref{Ass} and $y_N\les \ep_0$ in \emph{i)}, the above is bounded by
\begin{align*}
 \|\Lambda^N(\phi-Q)\|_{L^2}(1+y_{[N/2]+1})+\|\Lambda^{[N/2]}(\phi-Q)\|_{L^3}y_N
 \les \ep_0(1+y_N)\les \ep_0\,.
\end{align*}
Hence, we obtain the estimate \eqref{psi-by} for $|\nab|^{-1}Z^a\psi$. This completes the proof of the lemma.
\end{proof}

The above lemma and \eqref{Main-ini0} yield the bound for initial data in \eqref{sys-2}. Then we can state the result for \eqref{sys-2}-\eqref{constraints-re} as follows.

\begin{thm}\label{Ori_thm}
Let integer $N\geq 9$ and $0<\de\leq 1/8$ be two given constants and the initial data $(u_0,G_0,\psi_0)\in H^N_\Lambda$. Then there exists $\ep_0>0$ sufficiently small depending on $N$ and $\de$ such that, for all initial data $(u_0,G_0,\psi_0)$ satisfying the constriants \eqref{constraints-re} with
	\begin{equation}  \label{Main-ini}
		\|(u_0,G_0,\psi_0) \|_{H^N_\Lambda}\leq C_1\ep_0\,,
	\end{equation}
the evolutionary model of magnetoelasticity \eqref{sys-2}-\eqref{constraints-re} is global in time well-posed in three dimensions. Moreover, for all $t\in[0,+\infty)$ the solution $(u,G,\psi)\in H^{N}_Z$ has the bounds
	\begin{equation*}  
	E^{1/2}_N(u,G,\psi;t)\leq C_2C_1\ep_0\<t\>^\de\,,\qquad E^{1/2}_{N-2}(u,G,\psi;t)\leq C_2C_1 \ep_0\,.
	\end{equation*}
There exist $\GG^{(a)}_\infty\in L^2,\ \Psi^{(a)}_\infty\in H^{-1}$ for any $|a|\leq N-4$ such that
\begin{gather}  \label{Scatter-Gre}
\lim_{t\rightarrow \infty}\| Z^a u+\Im\big(e^{-it|\nab|}|\nab|^{-1}\d_j\mathcal G^{(a)}_{\infty,\cdot j}\big)\|_{L^2}+\|Z^a G-\Re \big(e^{-it|\nab|}\mathcal G^{(a)}_\infty\big)\|_{L^2}=0\,,\\\label{scattering}
\lim_{t\rightarrow \infty}\lV Z^a\psi-e^{-it\De}\Psi^{(a)}_\infty\rV_{H^{-1}}= 0\,.
\end{gather}
\end{thm}

The above result is obtained by the following proposition and continuity method.

\begin{prop}[Bootstrap Proposition]\label{Main_Prop}
    Let $\sqrt{\ep_0}\leq \de\leq 1/8$. Assume that $(u,G,\psi)$ is a solution to \eqref{sys-2}-\eqref{constraints-re} on some time interval $[0,T]$, $T\geq 1$ with initial data satisfying the assumption \eqref{Main-ini}.
    Assume also that the solution $(u,G,\psi)$ and the profile $\Psi^{(a)}=e^{it\De}Z^a\psi$ satisfy the bootstrap hypothesis
    \begin{gather}           \label{Main_Prop_Ass1}
        E_{N}^{1/2}\leq \ep_1\<t\>^{\de}\,,\qquad E_{N-2}^{1/2}\leq \ep_1\,,\\\label{Main_Prop_Ass2}
        X^{1/2}_N+Y^{1/2}_{N}\leq C_w\ep_1\<t\>^{\de}\,,\qquad X^{1/2}_{N-2}+Y^{1/2}_{N-2}\leq C_w\ep_1\,,
    \end{gather}
    where $\ep_1=C_2C_1\ep_0$ and $C_2,\ C_w$ are universal constants.
    Then the following improved bounds hold
    \begin{gather}\label{Main_Prop_result1}
    E_{N}^{1/2}\leq \frac{\ep_1}{2}\<t\>^{\de}\,,\qquad E_{N-2}^{1/2}\leq \frac{\ep_1}{2}\,,\\ \label{Main_Prop_result2}
    X^{1/2}_N+Y^{1/2}_{N}\leq \frac{C_w\ep_1}{2}\<t\>^{\de}\,,\qquad X^{1/2}_{N-2}+Y^{1/2}_{N-2}\leq \frac{C_w\ep_1}{2}\,.
    \end{gather}
\end{prop}

In the remaining sections, we focus on the proof of the Proposition \ref{Main_Prop}. Then we provide the proof of Theorem \ref{Ori_thm} and \ref{Ori_thm0}.

\medskip
\section{Decay estimates}\label{sec-dec}

In this section, we provide some weighted estimates and decay estimates, which will be frequently used in the later sections. First we state the weighted $L^{\infty}\mbox{-}L^2$ estimates and the bounds of $\<t\>f$, which is to prove the decays of elastic equations. Second, we review the basic decay estimate of Schr\"odinger equation. Then by bootstrap assumptions \eqref{Main_Prop_Ass1}-\eqref{Main_Prop_Ass2}, we get the decays of differential fields $\psi$, and hence the decay estimates of $A$, $A_{d+1}$, $v-v_\infty$ and $w-w_\infty$.
Here we start with the following weighted $L^{\infty}\mbox{-}L^2$ estimates.
\begin{lemma}
	Let $f\in H^2(\R^3)$, then there hold
	\begin{align} \label{ru}
		&\<r\>|f(x)|\lesssim \sum_{|\al|\leq 1}\|\d_r\Om^{\al}f\|_{L^2}^{1/2}\sum_{|\al|\leq 2}\|\Om^{\al}f\|_{L^2}^{1/2},\\\label{decay-1}
		&\<t\>\|f\|_{\lf(r<2\<t\>/3)}\les \|f\|_{L^2}+\|\<t-r\>\nab f\|_{L^2}+\|\<t-r\>\nab^2 f\|_{L^2}.
	\end{align}
	In particular, let $\om=x/|x|$, assume that $\div g=0$ for a vector $g=(g_1,g_2,g_3)$, then
	\begin{equation}  \label{omf}
		\|r^{3/2}(\om\cdot g)\|_{L^\infty}\lesssim \sum_{|\alpha|\leq 2}\lV \Om^{\alpha}g\rV_{L^2}.
	\end{equation}
\end{lemma}
\begin{proof}
 For \eqref{ru} and \eqref{decay-1}, one please refers to  \cite[Lemma 3.3]{Si} and \cite[Lemma 2.2]{CW}. For the last estimate \eqref{omf}, we can refer to \cite[Lemma 4.4]{LeiWang}.
\end{proof}

\begin{lemma} \label{Decay-phi}
	Assume that $(u,G,\psi)$ is the solution of \eqref{sys-2}. Then there hold
	\begin{align}  \label{AwayCone}
		&\<r\>| Z^a u|+\<r\>| Z^a G|\lesssim E^{1/2}_{|a|+2}\,,\\  \label{omuH}
		&\<r\>^{3/2}|\om\cdot Z^a u|+\<r\>^{3/2}|\om\cdot Z^a G^T|\les E^{1/2}_{|a|+2}\,,\\ \label{Decay-phi2}
        &\|Z^{a} u\|_{\lf}+\| Z^{a} G\|_{\lf}\les \<t\>^{-1}(E^{1/2}_{|a|+2}+X^{1/2}_{|a|+2})\,.
	\end{align}
\end{lemma}
\begin{proof}
The first bound follows from \eqref{ru}. The second one is obtained by $\div Z^au=0$, $\nab\cdot Z^a G^T=0$ and the bound \eqref{omf}. The last bound \eqref{Decay-phi2} follows from \eqref{decay-1} and \eqref{ru}.
\end{proof}

Next, we turn our attention to the decays of Schr\"odinger map flow.
Here we start with the basic decay estimates.

\begin{lemma}[\cite{HaMiNa99}, Lemma 2.4(1)]
    For any Schwartz function $f\in \mathcal{S}(\R^3)$ and $t>1$, we have
    \begin{align}   \label{df-decay}
        &\|\nab f\|_{L^6}\lesssim t^{-1}\| \Theta f\|_{L^2}\,.
    \end{align}
    where $\Theta =(x\c\nab-2it\De,\Om)$.
\end{lemma}
\begin{proof}
	By Sobolev embedding, we have
	\begin{align*}
		\|\nab f\|_{L^6}=\|e^{i\frac{x^2}{4t}}\nab f\|_{L^6}\les \|\nab (e^{i\frac{x^2}{4t}}\nab f)\|_{L^2}\les \sum_{j,k} \frac{1}{t}\|(x_j-2it\d_j)\d_k f\|_{L^2}\,.
	\end{align*}
Noting that by integration by parts, there holds 
\begin{align*}
	\sum_{j,k} \|(x_j-2it\d_j)\d_k f\|_{L^2}^2=\|(x\cdot\nab-2it\De)f\|_{L^2}^2+\sum_{j,k}\|\Om_{jk} f\|_{L^2}^2\,,
\end{align*}
which has been proved in \cite[Lemma 2.2]{Ha}. Hence, the bound \eqref{df-decay} follows.
\end{proof}

We then use this estimate to give the decays of differential fields $\psi$.
\begin{lemma}[Decays of $\psi$]  \label{Dec_psi-Lem}
    With the notations and hypothesis in Proposition \ref{Main_Prop}, for any $t\in[0,T]$ we have
	\begin{gather}     \label{Dec_psi}
	\lV Z^{N-1}\psi(t)\rV_{L^6} \lesssim \ep_1\<t\>^{-1+\de},\quad\qquad 
 \lV Z^{N-3}\psi(t)\rV_{L^6} \lesssim \ep_1\<t\>^{-1}\,.
	\end{gather}
\end{lemma}
\begin{proof}
    By \eqref{df-decay}, we have
    \begin{align*}
    \|Z^a\psi\|_{L^6}&\les \|\nab |\nab|^{-1}Z^a\psi\|_{L^6}
    \les t^{-1} \big(\|(x\c\nab-2it\De)|\nab|^{-1} Z^a\psi\|_{L^2}+\|\Om |\nab|^{-1}Z^a\psi\|_{L^2} \big)\,.
    \end{align*}
    In view of the relation $\Psi^{(a)}=e^{it\De}Z^a\psi$ and the commutator
    $(x\cdot\nab-2it\De)e^{-it\De}=e^{-it\De}(x\cdot\nab)$,
    we have for $t>1$
    \begin{align*}
        \|   Z^a\psi\|_{L^6}
        \lesssim &\  t^{-1} (\| e^{-it\De}( x\cdot \nab |\nab|^{-1}\Psi^{(a)})\|_{L^2}+E^{1/2}_{|a|+1})
        = t^{-1}(Y^{1/2}_{|a|+1}+E^{1/2}_{|a|+1})\,.
    \end{align*}
    Then the bounds in \eqref{Dec_psi} follow by \eqref{Main_Prop_Ass1} and \eqref{Main_Prop_Ass2}.
\end{proof}

Since the connections $A$ and $A_{d+1}$ satisfy some elliptic equations depending on $\psi$ and velocity $u$, then we can prove the following decays.
\begin{cor}[Decays of $A$ and $A_{d+1}$]  \label{Dec_A_Cor}
    With the notations and hypothesis in Proposition \ref{Main_Prop}, for the elliptic equations \eqref{Ell-A}, \eqref{Ell-Ad1} and any $t\in[0,T]$, we have
	\begin{align}  \label{Dec_d1/2-A}
    \| Z^N A\|_{L^2}&\lesssim \ep_1^{2}\<t\>^{-1/2+\de}\,,\\\label{Dec_dA}
    \| Z^{N-1} A\|_{L^\infty}&\lesssim \ep_1^2\<t\>^{-2+2\de}\,,\\   \label{Dec_Ad+1-L2}
    \| Z^N A_{d+1}\|_{L^2}&\lesssim \ep_1^2\<t\>^{-1+\de}\,,\\\label{Dec_Ad+1L3}
    \| Z^{N-2} A_{d+1}\|_{L^3\cap \lf}&\lesssim \ep_1^2\<t\>^{-2+\de}\,.
	\end{align}
\end{cor}
\begin{proof}
    We start with the first estimate \eqref{Dec_d1/2-A}. By \eqref{Ell-A} and Sobolev embedding we have
    \begin{align*}
        \| Z^N A\|_{L^2}&\lesssim \||\nab|^{-1}(Z^N\psi Z^{[N/2]}\psi)\|_{L^2}\les\|Z^N\psi\|_{L^2}\|Z^{[N/2]}\psi\|_{L^3}
        \les E^{1/2}_N E^{1/4}_{[N/2]}\|Z^{[N/2]}\psi\|_{L^6}^{1/2}\,.
    \end{align*}
Then by the decay of $\psi$ \eqref{Dec_psi} and \eqref{Main_Prop_Ass1}, we get the bound \eqref{Dec_d1/2-A}.

    For the second estimate \eqref{Dec_dA}, by Sobolev embeddings and H\"older's inequality, we get
    \begin{align*}
        \| Z^{N-1} A\|_{L^\infty}\lesssim &\ \| |\nab|^{-1}(Z^{N-1}\psi Z^{[(N-1)/2]]}\psi)\|_{\lf}
        \lesssim \|Z^{N-1}\psi Z^{[(N-1)/2]]}\psi\|_{L^{3-\de}\cap L^{3+\de}}\\
        \les &\ \|Z^{N-1}\psi\|_{L^6} \|Z^{[(N-1)/2]}\psi \|_{L^{\frac{6(3-\de)}{3+\de}}\cap L^{\frac{6(3+\de)}{3-\de}}}\,.
    \end{align*}
Using interpolations and \eqref{Dec_psi} to bound the above term by
\begin{align*}
\ep_1\<t\>^{-1+\de}  \big( \|Z^{[(N-1)/2]}\psi\|_{L^2}^{\frac{\de}{3-\de}}\|Z^{[(N-1)/2]}\psi\|_{L^6}^{1-\frac{\de}{3-\de}}+ \|Z^{[(N-1)/2]+1}\psi \|_{L^6}  \big)
\les \ \ep_1^2\<t\>^{-2+2\de}\,.
\end{align*}
Hence, the estimate \eqref{Dec_dA} is obtained.

    For the third bound \eqref{Dec_Ad+1-L2},
    from the equation \eqref{Ell-Ad1} we have
    \begin{align*}
        \| Z^N A_{d+1}\|_{L^2}\lesssim &\  \sum_{|b+c|=N}\|  Z^b\psi Z^c\psi\|_{L^2}+\sum_{|b+c+e|=N}\|Z^buZ^c\psi Z^e\psi\|_{L^{6/5}}\,.
    \end{align*}
    By \eqref{Main_Prop_Ass1} and \eqref{Dec_psi}, 
    we bound the first term in the right-hand side by $E_N^{1/2}\|Z^{[N/2]}\psi\|_{\lf}\les \ep_1^2\<t\>^{-1+\de}$.
    By Sobolev embeddings, \eqref{Main_Prop_Ass1}, \eqref{Dec_psi}, \eqref{Decay-phi}, the second term can be controlled by
    \begin{align*}
        &\ \| Z^N u\|_{L^2}\|Z^{[N/2]}\psi\|_{L^6}^2+\| Z^{[N/2]} u\|_{\lf}\|Z^{[N/2]}\psi\|_{L^3}\|Z^N\psi\|_{L^2}\\
        \les &\ \ep_1^3 \<t\>^{-2+\de}+\ep_1^3\<t\>^{-3/2+\de}\les \ep_1^3\<t\>^{-3/2+\de}\,.
    \end{align*}
Thus the estimate \eqref{Dec_Ad+1-L2} follows.

Finally, we prove the bound \eqref{Dec_Ad+1L3}. By Sobolev embeddings and \eqref{Dec_psi} we have
\begin{align*}
\sum_{|b+c|\leq N-2}\| \mathcal R (Z^b\psi Z^c\psi)\|_{L^3\cap L^\infty}\les&\  \| Z^{N-2}\psi\|_{L^6} \|Z^{[(N-2)/2]}\psi\|_{L^6}+\sum_{|b+c|\leq N-1}\| Z^b\psi Z^c\psi\|_{L^6}\\
\les &\ \|Z^{N-1}\psi\|_{L^6}\|Z^{[N/2]}\psi\|_{L^6\cap \lf}
\les  \ep_1^2 \<t\>^{-2+\de}\,.
\end{align*}
For the cubic term in \eqref{Ell-Ad1}, by Sobolev embeddings it suffices to bound 
\begin{align}  \label{Cubic}
        \sum_{|b+c+e|\leq N-2}\| Z^b u Z^c\psi Z^e\psi\|_{L^{3/2}}+\sum_{|b+c+e|\leq N-1}\| Z^b u Z^c\psi Z^e\psi\|_{L^2}\,.
\end{align}
Using H\"older's inequality, \eqref{Dec_psi} and \eqref{Decay-phi2} to bound 
\begin{align*}
	\eqref{Cubic}&\les \|Z^{N-1} u\|_{L^2}\|Z^{[N/2]}\psi\|_{L^6\cap L^\infty}\|Z^{[N/2]}\psi\|_{\lf}+\|Z^{[N/2]} u\|_{\lf}\|Z^{[N/2]}\psi\|_{L^6\cap L^\infty}\|Z^{N-1}\psi\|_{L^2}\\
	&\les  \ep_1\<t\>^\de \ep_1\<t\>^{-1}\ep_1\<t\>^{-1}+\ep_1\<t\>^{-1}\ep_1\<t\>^{-1}\ep_1\<t\>^{\de}\les \ep_1^3 \<t\>^{-2+\de}\,.
\end{align*}
Hence the bound \eqref{Dec_Ad+1L3} follows. This completes the proof of the lemma.
\end{proof}

We also have the decay estimates for the gauge $(v,\ w)$.
\begin{cor}
With the notations and hypothesis in Proposition \ref{Main_Prop}, for the Coulomb frame $v,w$, we denote $v_\infty=\lim_{|x|\rightarrow\infty}v(x)$ and $w_\infty=\lim_{|x|\rightarrow\infty}w(x)$. Then we have
    \begin{align}  \label{v-cong}
        \|v-v_\infty\|_{\lf}+\|w-w_\infty\|_{\lf}\les \ep_1\<t\>^{-\frac{1}{2}+\de}\,.
    \end{align}
\end{cor}
\begin{proof}
    By Sobolev embeddings, \eqref{phi-v-w} and $|\phi|=|w|=1$, we have
    \begin{align*}
        \|v-v_{\infty}\|_{\lf} &\les \|\nab v\|_{L^{3-\de'}\cap L^{3+\de'}}
        \les \||\psi|+|A|\|_{L^{3-\de'}\cap L^{3+\de'}}
        \les \|\psi\|_{L^{3-\de'}\cap L^{3+\de'}}+\|\<\nab\>A\|_{L^2}\,.
    \end{align*}
    By interpolation and \eqref{Dec_psi}, we bound $\|\psi\|_{L^{3-\de'}\cap L^{3+\de'}}\les \ep_1\<t\>^{-\frac{1}{2}+\frac{\de'}{3-\de'}}$, where we choose $\de'$ such that $\frac{\de'}{3-\de'}=\de$. From \eqref{Dec_d1/2-A}, the term $\|\<\nab\>A\|_{L^2}$ is bounded by $\ep_1\<t\>^{-1/2+\de}$.
Thus the bound \eqref{v-cong} for $v-v_\infty$ follows.
The second term in \eqref{v-cong} can also be proved in a same way.
\end{proof}

\medskip
\section{Energy estimates}  \label{sec-EnEs}
In this section, we focus on the following energy estimates. 

\begin{prop}     \label{Prop_va}
Let $N\geq 9$. Assume that $(u,G,\psi)$ is the solution of \eqref{sys-2}-\eqref{constraints-re} satisfying the assumptions \eqref{Main_Prop_Ass1} and \eqref{Main_Prop_Ass2}. Then for any $t\in[0,T]$, we have the energy bounds:	
\begin{align}       \label{Ev}
\frac{d}{dt}E_N(t)\leq C_E \ep_1^3\<t\>^{-1+2\de}\,,\qquad 
\frac{d}{dt} E_{N-2}(t)\leq C_E \ep_1^3\<t\>^{-4/3+2\de}\,.
\end{align}
\end{prop}

\subsection{Higher-order energy estimates}\label{sec-HighEnergy}
This section is devoted to the higher-order energy estimate in \eqref{Ev},
which is obtained by the following two energy estimates of $(u,G,\psi)$: for any $|a|\leq N$,
\begin{align}\label{Eu-2}
	\frac{1}{2}\frac{d}{dt}\int  (|Z^a u|^2+| Z^a G|^2)\ dx&\leq  C\ep_1^3 \<t\>^{-1+2\de}- \int  Z^a u\cdot   \Re(\psi\overline{\d_j Z^a\psi_j})\ dx\,,\\ \label{Ephi}
        \frac{1}{2}\frac{d}{dt}\int | |\nab|^{-1} Z^a \psi|^2+| Z^a \psi|^2\ dx
		&\leq C \ep_1^3\<t\>^{-1+2\de}+ \int  Z^a u\cdot   \Re(\psi\overline{\d_j Z^a\psi_j})\ dx\,.
\end{align}

\begin{proof}[Proof of the energy estimate \eqref{Eu-2}]
By the first two equations in (\ref{sys-VF}) and $\div Z^a u=0$, we calculate
	\begin{align*}
		\frac{1}{2}\frac{d}{dt}\int  (|Z^a u|^2+| Z^a G|^2)\ dx
		=\int (Z^au\cdot f_a+ Z^aG_{ij} g_{a,ij})\ dx\,.
	\end{align*}
Then from $f_a,\ g_a$ in \eqref{fa} and \eqref{ga}, it suffices to bound
\begin{align*}
	I_1&: =\sum_{b+c=a}C_a^b \int  Z^a u \big(-Z^b u\cdot \nab Z^c u +\nab\cdot(Z^b G Z^c G^T)\big)\ dx,\\
	I_2&: =- \sum_{b+c=a}C_a^b\int  Z^a u\cdot \d_j  \Re( Z^b\psi  Z^c\bar{\psi_j})\ dx,\\
	I_3&:= \sum_{b+c=a}C_a^b \int  Z^a G_{ij}\cdot \big(-Z^b u\cdot \nab Z^c G_{ij}+\d_k Z^b u_i Z^c G_{kj}\big)\ dx.
\end{align*}

\emph{1) We prove the estimates of $I_1$ and $I_3$: $|I_1|+|I_3|\les \ep_1^3\<t\>^{-1+2\de}$.}

We use $\div u=0$ and $\d_j G_{jk}=0$ to rearrange the $I_1$ and $I_3$ as
\begin{equation}   \label{I1I4}
\begin{aligned} 
	I_1+I_3&= \sum_{b+c=a;|c|<| a|}C_a^b \int  Z^a u_i \big(-Z^b u\cdot \nab Z^c u_i + \d_j Z^c G_{ik} Z^b G_{jk}\big)\ dx\\
	&\quad +\sum_{b+c=a:|c|<|a|}C_a^b \int  Z^a G_{ij}\cdot \big(-Z^b u\cdot \nab Z^c G_{ij}+\d_k Z^c u_i Z^b G_{kj}\big)\ dx,
\end{aligned}
\end{equation}
then by \eqref{Decay-phi2} and \eqref{Main_Prop_Ass2}, which can be bounded by $E_N\|Z^{[N/2]+1}(u,G)\|_{L^\infty}\les \ep_1^3\<t\>^{-1+2\de}$.

\emph{2) We prove the estimate of $I_2$: 
\begin{equation*}
    |I_2|\leq C\ep_1^3\<t\>^{-1+2\de}-\int Z^au\cdot\Re(\psi \overline{\d_j Z^a\psi_j})dx.
\end{equation*}}

By the constraint \eqref{comm}, we rewrite the integral $I_2$ as 
\begin{align} \nonumber
	I_2&=- \sum_{b+c=a}C_a^b\int  Z^a u\cdot  \Re( \d_j  Z^b\psi\   \overline{Z^c\psi_j}+ Z^b\psi \  \overline{\d_j Z^c  \psi_j})\ dx\\ \nonumber
	&= - \sum_{b+c=a}C_a^b\int  Z^a u\cdot  \Re\big( \d  Z^b\psi_j\   \overline{Z^c\psi_j}-i \sum_{b_1+b_2=b}C_b^{b_1} (Z^{b_1} A_j Z^{b_2}\psi-Z^{b_1} A Z^{b_2}\psi_j) \overline{Z^c\psi_j}\\\nonumber
	&\qquad + Z^b\psi \  \overline{\d_j Z^c  \psi_j}\big)\ dx\,.
\end{align}
Using integration by parts, this is further rearranged as 
\begin{align} \nonumber
	I_2&\leq - \frac{1}{2}\sum_{b+c=a}C_a^b \int Z^a u\cdot  \nab \Re (  Z^b\psi_j\cdot \overline{ Z^c\psi_j})\ dx
	+ C\sum_{b_1+b_2+c=a}\int | Z^a u| | Z^{b_1} A_j| |Z^{b_2}\psi| |Z^c\psi_j|dx\\ \label{I2}
 &\quad + C\sum_{b+c=a;|c|<|a|}\int  |Z^a u||Z^b\psi| |\d_j Z^c\psi_j|dx
	 - \int  Z^a u\cdot   \Re(\psi\ \overline{\d_j Z^a\psi_j})\ dx\\\nonumber  
	:&= I_{2,1}+I_{2,2}+I_{2,3}+I_{2,4}\,.
\end{align}
The last integral $I_{2,4}$ is retained temporarily here. By integration by parts and $\div Z^a u=0$, the term $I_{2,1}$ vanishes. For the second integral $I_{2,2}$, by Sobolev embedding, \eqref{Main_Prop_Ass1}, \eqref{Dec_d1/2-A}, \eqref{Dec_psi} we have
\begin{align}   \label{I22}
	I_{2,2}\les \|Z^N u\|_{L^2}\|Z^NA\|_{L^2}\|Z^N\psi\|_{L^2}\|Z^{[N/2]}\psi\|_{L^\infty}\les \ep_1^5\<t\>^{-3/2+3\de}\,.
\end{align}
From \eqref{Decay-phi2} and \eqref{Main_Prop_Ass2}, the third integral $I_{2,3}$ is bounded by $E_N\|Z^{[N/2]+1}\psi\|_{L^\infty}\les \ep_1^3\<t\>^{-1+2\de}$.
Collecting the above estimates, we obtain the estimate of $I_2$. This completes the proof of \eqref{Eu-2}.
\end{proof}

\begin{proof}[Proof of the energy estimate \eqref{Ephi}]
By the $Z^a\psi$-equation in (\ref{sys-VF}) and integration by parts, we calculate
\begin{align}  \label{Ene-psi}
\frac{1}{2}\frac{d}{dt}\int | |\nab|^{-1} Z^a \psi|^2+| Z^a \psi|^2\ dx
=\int \Im\big(|\nab|^{-1}h_a\cdot |\nab|^{-1}\overline{Z^a\psi} \big)+\Im\big(h_a\cdot \overline{Z^a\psi} \big)\ dx\,.
\end{align}
From the expression of $h_a$ in \eqref{ha}, we decompose the above integral into three terms 
\begin{equation*}
	\eqref{Ene-psi}=\sum_{j=1}^3\int \Im\big(|\nab|^{-1}h_{a,j}\cdot |\nab|^{-1}\overline{Z^a\psi} \big)+\Im\big(h_{a,j}\cdot \overline{Z^a\psi} \big)\ dx:=J_1+J_2+J_3\,,
\end{equation*}
where $h_{a,j}$ for $j=1,2,3$ are the nonlinear terms in \eqref{ha} given by
\begin{equation}\label{has}
\begin{aligned}
h_{a,1}:&=-i\sum_{b+c=a} C_a^b Z^b (u-2A)\cdot \nab Z^c\psi\,,
\qquad h_{a,2}:=-i\sum_{b+c=a} C_a^b \nab Z^b u\cdot Z^c\psi\,,\\
h_{a,3}:&=\sum_{b+c=a}C_a^b Z^b A_{d+1} Z^c\psi+\sum_{b+c+e=a} C_a^{b,c} \big(Z^b (u-A)\cdot Z^c A Z^e\psi-i\Im(Z^b \psi_l Z^c\bar{\psi})Z^e\psi_l\big)\,.
\end{aligned}	
\end{equation}

\emph{1). We prove the estimate of $J_1$: $J_1\les \ep_1^3\<t\>^{-1+2\de}$.}

From the expression of $h_{a,1}$, by $\div Z^b u=\div Z^b A=0$ and integration by parts, we write $J_1$ as
\begin{align}\nonumber
J_1&= \int \Im\big(|\nab|^{-1}\nab\cdot(-i\sum_{b+c=a} C_a^b Z^b (u-2A) Z^c\psi)\cdot |\nab|^{-1}\overline{Z^a\psi} \big)dx\\ \label{J1}
&\quad +\int \Im\big((-i\sum_{b+c=a;|c|<N} C_a^b Z^b (u-2A)\cdot \nab Z^c\psi)\cdot \overline{Z^a\psi} \big)\ dx:=J_{1,1}+J_{1,2}.
\end{align}
Then from \eqref{Decay-phi2}, \eqref{Dec_dA}, \eqref{Main_Prop_Ass1}, \eqref{Dec_d1/2-A} and \eqref{Dec_psi}, we have
\begin{align*}
	J_{1,1}\les &\  \big(\|(Z^{[N/2]}u,Z^{[N/2]}A)\|_{L^\infty}\|Z^N \psi\|_{L^2} +\|(Z^{N}u,Z^{N}A)\|_{L^2}\|Z^{[N/2]}\psi\|_{\lf}\big)\||\nab|^{-1}Z^a\psi\|_{L^2} \\
	\les &\ \big( (\ep_1\<t\>^{-1}+\ep_1^2\<t\>^{-2+2\de})\ep_1\<t\>^{\de}+(\ep_1\<t\>^{\de}+\ep_1^2\<t\>^{-1/2+\de})\ep_1\<t\>^{-1}  \big)\ep_1\<t\>^\de
	\les \ep_1^3\<t\>^{-1+2\de}\,.
\end{align*}
The second integral $J_{1,2}$ can also be controlled by $\ep_1^3\<t\>^{-1+2\de}$ similarly.
Hence, the estimate of $J_1$ follows.

\emph{2). We prove the estimate of $J_2$: \begin{equation*}
    J_2\leq C\ep_1^3\<t\>^{-1+2\de}+\int Z^a u\cdot\Re (  \psi \d_j\overline{Z^a\psi_j} )\ dx.
\end{equation*}
}

Note that the expression of $h_{a,2}$ in \eqref{has}, we rewrite the $J_2$ as
\begin{align}  \nonumber
J_2&=\int \Im\big(|\nab|^{-1}(-i\sum_{b+c=a} C_a^b \d_j Z^b u\cdot Z^c\psi))\cdot |\nab|^{-1}\overline{Z^a\psi_j} \big)\\  \label{J2}
&\quad +\int \Im\big((-i\sum_{b+c=a;|b|<N} C_a^b \d_j Z^b u\cdot Z^c\psi)\cdot \overline{Z^a\psi_j} \big)\ dx+\int \Im\big((-i\d_j Z^a u\cdot \psi)\cdot \overline{Z^a\psi_j} \big)\ dx\\\nonumber
&= J_{2,1}+J_{2,2}+J_{2,3}\,.
\end{align}

Here we must rewrite the term $J_{2,1}$ to gain more decay. Precisely, by integration by parts and \eqref{comm}, we have
\begin{align}   \nonumber
    J_{2,1}
    &=-\sum_{b+c=a} C_a^b\int \Re \Big(\big(|\nab|^{-1}\d_j(  Z^b u\cdot Z^c\psi)-|\nab|^{-1}(Z^b u_k\cdot \d_j Z^c\psi_k)\big) |\nab|^{-1}\overline{Z^a\psi_j} \Big)\ dx\\\label{J21}
    &\leq C\sum_{b+c=a}\|Z^b u Z^c\psi\|_{L^2} \||\nab|^{-1}Z^a\psi\|_{L^2}\\\nonumber
    &\quad + \sum_{b+c=a} C_a^b\int \Re \Big(  |\nab|^{-1}(Z^b u_k \d_k Z^c\psi_j)\cdot |\nab|^{-1}\overline{Z^a\psi_j} \Big)\ dx\\\nonumber
    &\quad + C\sum_{b+c_1+c_2=a}\int \big||\nab|^{-1}(Z^b u\ Z^{c_1}A\ Z^{c_2}\psi)\cdot |\nab|^{-1}\overline{Z^a\psi}\big|\ dx:=\mathcal R_1+\mathcal R_2+\mathcal R_3\,.
\end{align}
By \eqref{Decay-phi2}, \eqref{Dec_psi}, \eqref{Main_Prop_Ass1} and \eqref{Main_Prop_Ass2}, we bound the first term $\mathcal R_1$ by
\begin{align}   \label{J21a}
	&\mathcal R_1	\les (\|Z^{[N/2]}u\|_{\lf}+\|Z^{[N/2]}\psi\|_{\lf})E_N
	\les \ep_1^3 \<t\>^{-1+2\de}\,.
\end{align}
For the second term $\mathcal R_2$, by $\div Z^b u=0$, we can move the $\d_k$ out of $Z^b u_k \d_k Z^c\psi_j$, then we obtain
\begin{align*}
	\mathcal R_2=&\ \sum_{b+c=a} C_a^b\int \Re \Big(  |\nab|^{-1}\d_k(Z^b u_k  Z^c\psi_j)\cdot |\nab|^{-1}\overline{Z^a\psi_j} \Big)\ dx\\
	\les &\ \sum_{b+c=a}\|Z^b u Z^c\psi\|_{L^2} \||\nab|^{-1}Z^a\psi\|_{L^2}\les \ep_1^3 \<t\>^{-1+2\de}\,.
\end{align*}
Using Sobolev embeddings, \eqref{Main_Prop_Ass1}, \eqref{Dec_d1/2-A}, \eqref{Decay-phi2} and \eqref{Dec_psi} we bound the third term $\mathcal R_3$ by 
\begin{align*}
\mathcal R_3\les &\  \sum_{b+c_1+c_2=a}\|Z^b u\ Z^{c_1}A\  Z^{c_2}\psi\|_{L^{6/5}}\ep_1\<t\>^\de\\
    \les &\ \|Z^{N}A\|_{L^2}(\|Z^N u\|_{L^2}\|Z^{[N/2]}\psi\|_{\lf}+\|Z^{[N/2]} u\|_{L^\infty}\|Z^{N}\psi\|_{L^2})\ep_1\<t\>^\de\\
    \les &\ \ep_1^2\<t\>^{-1/2+\de}\ep_1\<t\>^{-1}\ep_1^2\<t\>^{2\de}
    \les \ep_1^5 \<t\>^{-3/2+3\de}\,.
\end{align*}
Hence, the term $J_{2,1}$ is controlled by $\ep_1^3\<t\>^{-1+2\de}$.

Similar to the estimate \eqref{J21a}, we bound the $J_{2,2}$ by $\ep_1^3\<t\>^{-1+2\de}$.
By integration by parts and \eqref{Dec_psi}, the last integral $J_{2,3}$ is rewritten as 
\begin{align*}
	J_{2,3}&=\int \Re \big( Z^a u\cdot \d_j\psi\overline{Z^a\psi_j} \big)\ dx+\int \Re \big( Z^a u\cdot \psi \d_j\overline{Z^a\psi_j} \big)\ dx\\
	&\leq C \|Z^a u\|_{L^2} \|\nab\psi\|_{\lf}\|Z^a\psi\|_{L^2}+\int \Re \big( Z^a u\cdot \psi \d_j\overline{Z^a\psi_j} \big)\ dx\\
	&\leq C\ep_1^3\<t\>^{-1+2\de}+\int \Re \big( Z^a u\cdot \psi \d_j\overline{Z^a\psi_j} \big)\ dx\,.
\end{align*}
Therefore, the estimate of $J_2$ follows.

\emph{3). We prove the estimate of $J_3$: $J_3\les \ep_1^3\<t\>^{-3/2+2\de}$.}

By H\"older's inequality and Sobolev embedding, we have
\begin{align*}
    J_3&\les \||\nab|^{-1}h_{a,3}\|_{H^1} \||\nab|^{-1}Z^a\psi\|_{H^1}\les \ep_1\<t\>^\de\|h_{a,3}\|_{L^2\cap L^{6/5}}\,.
\end{align*}
By Sobolev embedding, \eqref{Dec_Ad+1L3}, \eqref{Main_Prop_Ass1}, \eqref{Dec_Ad+1-L2} and \eqref{Dec_psi}, we bound the quadratic term in $h_{a,3}$ by
\begin{align*}
	\sum_{b+c=a}\|Z^b A_{d+1}Z^c\psi\|_{L^2\cap L^{6/5}}&\les \|Z^{[N/2]}A_{d+1}\|_{L^3\cap\lf}\|Z^N\psi\|_{L^2}+\|Z^N A_{d+1}\|_{L^2}\|Z^{[N/2]}\psi\|_{L^3\cap \lf}\\
	&\les \ep_1^2 \<t\>^{-2+\de}\ep_1\<t\>^\de+\ep_1^2\<t\>^{-1+\de}(\ep_1\<t\>^{-1/2}+\ep_1\<t\>^{-1})\les \ep_1^3\<t\>^{-3/2+\de}\,.
\end{align*}
By Sobolev embeddings, \eqref{Dec_d1/2-A}, \eqref{Dec_psi}, \eqref{Dec_dA} and \eqref{Main_Prop_Ass1}, the first cubic term in $h_{a,3}$ is bounded by
\begin{align*}
	&\sum_{b+c+e=a}\|Z^b(u-A)Z^c A Z^e\psi\|_{L^2\cap L^{6/5}}\\
	\les &\ \big\||Z^Nu|+|Z^NA|\big\|_{L^2}(\|Z^{N}A\|_{L^2}\|Z^{[N/2]}\psi\|_{\lf}+\|Z^{[N/2]}A\|_{\lf}\|Z^{N}\psi\|_{L^2})\\
	\les &\ \ep_1\<t\>^\de (\ep_1^2\<t\>^{-1/2+\de}\ep_1\<t\>^{-1}+\ep_1^2\<t\>^{-2+2\de}\ep_1\<t\>^{\de})\les \ep_1^4 \<t\>^{-3/2+2\de}\,.
\end{align*}
The second cubic term can also be estimated by
\begin{align*}
	&\sum_{b+c+e=a}\|Z^b\psi Z^c \psi Z^e\psi\|_{L^2\cap L^{6/5}}
	\les \|Z^N\psi\|_{L^2}\|Z^{[N/2]}\psi\|_{L^6\cap \lf}^2\les \ep_1\<t\>^\de \ep_1^2\<t\>^{-2}\les \ep_1^3\<t\>^{-2+\de}\,.
\end{align*}
Hence, the term $J_3$ is controlled by $\ep_1^3\<t\>^{-3/2+2\de}$. We complete the proof of estimate \eqref{Ephi}.
\end{proof}

\subsection{Lower-order energy estimates}\label{sec-LowEnergy}

This section is devoted to the lower-order energy estimate in \eqref{Ev}.
\begin{proof}[Proof of \eqref{Ev}]

\emph{Step 1: We prove that}
\begin{align}   \label{E1}
    \frac{d}{dt}\sum_{|a|\leq N-2}\Big(\lV  Z^a u\rV_{L^2}^2+\| Z^aG\|_{L^2}^2\Big)\les \ep_1^3 \<t\>^{-3/2+\de}\,.
\end{align}

By the first two equations in (\ref{sys-VF}), we calculate
\begin{align*}
	\frac{1}{2}\frac{d}{dt}\Big(\lV  Z^a u\rV_{L^2}^2+\| Z^aG\|_{L^2}^2\Big)
	&=\sum_{b+c=a}C_a^b \int Z^a u \big(-Z^b u\cdot \nab Z^c u +\nab \cdot (Z^b G Z^cG^T)\big)\\
	&\quad +Z^a G_{ij}\cdot \big(-Z^b u\cdot\nab  Z^c G_{ij}+\d_k Z^b u_i Z^c G_{kj}\big)\ dx\\
	&\quad 	- \sum_{b+c=a}C_a^b\int  Z^a u\ \d_j \Re ( Z^b\psi\cdot \overline{ Z^c\psi_j})\ dx
	:=K_1+K_2\,.
\end{align*}
Then we proceed to bound the terms $K_1$ and $K_2$ separately.

\medskip
\emph{1) Estimate of $K_1$.}

Similar to \eqref{I1I4}, using $\div u=0$, $\d_j G_{jk}=0$ and integration by parts, we rewrite the $K_1$ as
\begin{align*}
	K_1&= \sum_{b+c=a;|c|<| a|}C_a^b \int  Z^a u_i \big(-Z^b u\cdot \nab Z^c u_i + \d_j Z^c G_{ik} Z^b G_{jk}\big)\ dx\\
	&\quad +\sum_{b+c=a:|c|<|a|}C_a^b \int   Z^a G_{ij}\cdot \big(-Z^b u\cdot \nab Z^c G_{ij}+\d_k Z^c u_i Z^b G_{kj}\big)\ dx\,.
\end{align*}
 Thanks to the constriants $\nab \cdot Z^b u=\nab\cdot Z^b G^T=0$, all of the above terms are dealt with in similar method. Here we only consider the following term in detail
\begin{align*}
	k_1:=\sum_{|a|\leq N-2}\sum_{b+c=a;|c|<| a|}|\d_j Z^c G_{ik} Z^b G_{jk}|\,.
\end{align*}
 In the region $r<2\<t\>/3$, applying \eqref{Decay-phi2} and \eqref{Main_Prop_Ass2}, we infer
\begin{equation}   \label{k1-1}
\begin{aligned}
	\|k_1\|_{L^2(r<2\<t\>/3)}&\les \<t\>^{-1}\|\<t-r\>\nab Z^{N-3} G\|_{L^2}\|Z^{N-2} G \|_{\lf} \\
	&\les \<t\>^{-1} X^{1/2}_{N-2}\ep_1\<t\>^{-1+\de} \les \ep_1^2\<t\>^{-2+\de}\,.
\end{aligned}
\end{equation}
In the region $r\geq 2\<t\>/3$, we utilize $\nab=\om\d_r-\frac{\om\times \Om}{r}$ to rewrite $k_1$ as
\begin{align*}
	k_1=\sum_{|a|\leq N-2}\sum_{b+c=a;|c|<| a|}\Big|\Big(\om_j \d_r-\frac{(\om\times \Om)_j}{r} \Big)Z^c G_{ik} Z^b G_{jk}\Big|\,.
\end{align*}
Then it follows from \eqref{omuH}, \eqref{Decay-phi2} and \eqref{Main_Prop_Ass2} that
\begin{align}  \nonumber 
	\|k_1\|_{L^2(r\geq 2\<t\>/3)}&\les \<t\>^{-3/2}\|\d_r Z^{N-3}G\|_{L^2}\|r^{3/2}\om_j Z^{N-2}G_{jk}\|_{\lf}\\ \nonumber
     &\quad  +\<t\>^{-1}\|\Om Z^{N-3}G\|_{\lf}\|Z^{N-2}G\|_{L^2}\\ \label{k1-2}
	&\les \<t\>^{-3/2}E^{1/2}_{N-2}E^{1/2}_N
+\<t\>^{-2}\big(E^{1/2}_{N}+X^{1/2}_N\big)E^{1/2}_{N-2}\les \ep_1^2\<t\>^{-3/2+\de}\,.
\end{align}
In a similar way, we can also bound the other three terms in $K_1$. Hence, we bound $K_1$ by $\ep_1^3\<t\>^{-3/2+\de}$.

\smallskip 
\emph{2) Estimate of $K_2$.}

Along the same way in the calculation \eqref{I2} of $I_2$, we also have
\begin{align*}
	K_2&\leq C\sum_{b_1+b_2+c=a} \int |Z^a u||Z^{b_1}A||Z^{b_2}\psi| |Z^c\psi|dx-\sum_{b+c=a}C_a^b\int Z^a u\cdot \Re(Z^b\psi \overline{\d_j Z^c\psi_j}) dx\\
&:= K_{2,1}+K_{2,2}\,.
\end{align*}
The first integral $K_{2,1}$ is bounded by $\ep_1^5\<t\>^{-3/2+3\de}$, which is estimated as the case \eqref{I22}.
For the integral $K_{2,2}$, we have from \eqref{Main_Prop_Ass1} and \eqref{Dec_psi} that
\begin{align*}
	K_{2,2}&\les \|Z^{N-2} u\|_{L^2} \|Z^{N-1}\psi\|_{L^6}\|  Z^{[N/2]+1}\psi\|_{L^3}
 \les \ep_1 \ep_1\<t\>^{-1+\de}\ep_1\<t\>^{-1/2}
\les\ep_1^3 \<t\>^{-3/2+\de}\,.
\end{align*}	
Hence, the term $K_2$ is bounded by $\ep_1^3\<t\>^{-3/2+\de}$.
This completes the proof of the estimate \eqref{E1}.

\medskip
\emph{Step 2: We prove that}
\begin{align}  \label{E2}
    \frac{d}{dt}\sum_{|a|\leq N-2}\int | |\nab|^{-1} Z^a \psi|^2+| Z^a \psi|^2\ dx\les \ep_1^3\<t\>^{-4/3+2\de}\,.
\end{align}

From the $ Z^a\psi$-equation in (\ref{sys-VF}) and $\div u=0$, it suffices to bound
\begin{align*}
L_j:=\int \Im\big(|\nab|^{-1}h_{a,j}\cdot |\nab|^{-1}\overline{Z^a\psi} \big)+\Im\big(h_{a,j}\cdot \overline{Z^a\psi} \big)\ dx\,,
\end{align*}
where $h_{a,j}$ for $j=1,2,3$ are given in \eqref{has}.

\emph{1) Estimate of $L_1$.}
Similar to the estimates of $J_1$ in \eqref{J1}, it suffices to bound 
\begin{align*}\nonumber
L_1
&\les \|Z^{N-2}(u-2A)\|_{L^3}\|Z^{N-2}\psi\|_{L^6}E^{1/2}_{N-2}\,.
\end{align*}
From the decays \eqref{Decay-phi2}, \eqref{Dec_dA} and \eqref{Dec_psi}, we obtain
\begin{align*}
    L_1&\les \|Z^{N-2}(u-2A)\|_{L^2}^{2/3}\|Z^{N-2}(u-2A)\|_{\lf}^{1/3} \ep_1\<t\>^{-1+\de}\ep_1\\
    &\les \ep_1\<t\>^{-1/3+\de/3}\ep_1^2\<t\>^{-1+\de}\les \ep_1 \les \ep_1^3\<t\>^{-4/3+2\de}\,.
\end{align*}

\emph{2) Estimate of $L_2$.}

By the calculations in \eqref{J2} and \eqref{J21}, it suffices to bound
\begin{align*}
    L_2&\les \sum_{|b+c|=|a|\leq N-2}(\|Z^b u Z^c\psi\|_{L^2} \||\nab|^{-1}Z^a\psi\|_{L^2}+\|\nab Z^b u\|_{L^2} \|Z^c\psi Z^a\psi\|_{L^2})\\
    &\quad + \sum_{|b+c_1+c_2|=|a|\leq N-2}\||\nab|^{-1}(Z^b u Z^{c_1}A Z^{c_2}\psi)\|_{L^2}\||\nab|^{-1}Z^a\psi\|_{L^2}:=L_{2,1}+L_{2,2}.
\end{align*}
Then from \eqref{Decay-phi2}, \eqref{Dec_psi} and \eqref{Main_Prop_Ass1}, the first term in the right-hand side
\begin{align*}
	L_{2,1}&\les \|Z^{N-2} u\|_{L^3}\|Z^{N-2} \psi\|_{L^6}\ep_1+\|\nab Z^{N-2} u\|_{L^2}\|Z^{N-2}\psi\|_{L^3}\| Z^{N-2}\psi\|_{L^6}\\
	&\les \ep_1\<t\>^{-1/3+\de/3}\ep_1\<t\>^{-1+\de}\ep_1+ \ep_1\<t\>^{\de}\ep_1\<t\>^{-1/2+\de/2}\ep_1\<t\>^{-1+\de}
	\les \ep_1^3\<t\>^{-4/3+2\de}\,.
\end{align*}
Using \eqref{Decay-phi2}, \eqref{Dec_d1/2-A} and \eqref{Dec_psi}, we bound the second term by
\begin{align*}
    L_{2,2}&\les \|Z^{N-2}u\|_{\lf}\|Z^{N-2} A\|_{L^2}\|Z^{N-2}\psi\|_{L^3}\ep_1\\
    &\les \ep_1\<t\>^{-1+\de}\ep_1\<t\>^{-1/2+\de}\ep_1\<t\>^{-1/2+\de/2}\ep_1\les \ep_1^4\<t\>^{-2+3\de}\,.
\end{align*}
Hence, the $L_2$ is bounded by $\ep_1^3\<t\>^{-4/3+2\de}$. Finally, along the same line as $J_3$ above, we can obtain $|L_3|\les \ep_1^3 \<t\>^{-3/2+3\de}$.
This completes the proof of estimate \eqref{E2}.
\end{proof}

\medskip 
\section{Estimates of the \texorpdfstring{$L^2$}{} Weighted Norms}\label{sec-L2}

In this section, we focus on the $L^2$ weighted estimates for $X_j$ and $Y_j$ under the bootstrap assumptions \eqref{Main_Prop_Ass1}, \eqref{Main_Prop_Ass2} and decay estimates in Section \ref{sec-dec}.

\subsection{Weighted estimates of elastic equations}
Here we focus on the estimates of $L^2$ weighted norm $X_{j}$ \eqref{XX}. 
The main proposition we will prove is as follows.
\begin{prop}\label{XX-prop}
Let $N\geq 9$. Suppose that $(u,G,\psi)$ is a solution of \eqref{sys-2}-\eqref{constraints-re} satisfying the assumption \eqref{Main_Prop_Ass1} and \eqref{Main_Prop_Ass2}. Then for any $k\leq N$, we have
	\begin{align}\label{L2uH}
		&X_{k}\leq C_3 \ep_1 X_k+C_3E_k\, .
	\end{align}
\end{prop}

Here we start with some useful lemmas.
\begin{lemma}  \label{Qk}
	Assume that $Z^a G$ satisfies the constraint \eqref{curl-vf}, then there holds
	\begin{align}   \label{H-(t-r)}
		\|(t-r)\nab Z^a G\|_{L^2}^2\leq 4 \|Z^a G\|_{L^2}^2+C\|(t-r)\nab \cdot Z^a G\|_{L^2}^2+C\QQ_a\,,
	\end{align}
where
\begin{equation*}
\QQ_a=\sum_{b+c=a}\int (t-r)^2 \d_j Z^a G_{ik}\big(Z^b G_{lk}\d_l Z^c G_{ij}-Z^b G_{lj}\d_lZ^c G_{ik}\big)\ dx \,.
\end{equation*}
Moreover,
\begin{align}       \label{Qa-est}
	\QQ_a\les  X_{|a|+1} E_{[|a|/2]+2}^{1/2}+X_{|a|+1}^{1/2}E^{1/2}_{|a|}
X_{[|a|/2]+3}^{1/2}\,.
\end{align}
\end{lemma}
\begin{proof}
The formula \eqref{H-(t-r)} has been proved in \cite[Lemma 6.2]{LeiWang}, which is obtained from \eqref{curl-vf} and integration by parts.
For the estimate to formula $\QQ_a$, by the Sobolev embeddings, \eqref{AwayCone} and \eqref{Decay-phi2}, we have
	\begin{align*}
		\QQ_a&\les \int \<t-r\>^2 |\nab Z^{a}G|\Big(|Z^{[|a|/2]}G\nab Z^{|a|}G|+|Z^{|a|}G\nab Z^{[|a|/2]}G|\Big)\ dx\\
		&\les \|\<t-r\>\nab Z^{a}G\|_{L^2}^2\||Z^{[|a|/2]}G\|_{\lf}
		 +\|\<t-r\>\nab Z^{a}G\|_{L^2}\| Z^{|a|}G\|_{L^2}\|\<t-r\>\nab Z^{[|a|/2]}G\|_{\lf}\\
		&\les X_{|a|+1} E_{[|a|/2]+2}^{1/2}+X_{|a|+1}^{1/2}E^{1/2}_{|a|}
X_{[|a|/2]+3}^{1/2}\,.
	\end{align*}
This completes the proof of Lemma \ref{Qk}.
\end{proof}

\begin{lemma}\label{Lem5.5}
	Assume that $(u,G,\psi)$ is the solution of \eqref{sys-VF} with nonlinearities $f_a,\ g_a$ and $\mathcal N_a$ given in \eqref{fa}, \eqref{ga} and \eqref{Na}, respectively.  We denote
	\begin{align}   \label{Nk}
		\mathcal L_k&:=\sum_{|a|\leq k}(|Z^a u|+|Z^a G|)\,,\qquad 
		N_k:=\sum_{|a|\leq k-1}\big(t|f_a|+t|g_a|+(t+r)|\mathcal N_a|+t|\nab Z^a p|\big)\,.
	\end{align}
Then for all $|a|\leq k-1$,
\begin{align}\label{goodun}
	(t\pm r)|\nab Z^a u\pm \nab \cdot Z^a G\otimes \om |\les \mathcal L_k+N_k\,.
\end{align}
\end{lemma}

\begin{proof}
For the proof of the lemma, we can refer to Lemma 6.3 in \cite{LeiWang}.
\end{proof}

\begin{lemma}\label{NK0}
	Let $N_k$ be the nonlinear term  \eqref{Nk}. Then we have
	\begin{align}     \label{Nk-est}
			\|N_k\|_{L^2}&\les \ep_1 E_k^{1/2}\,.
	\end{align}
\end{lemma}
\begin{proof}
Applying the operator $\nab\cdot$ to the $Z^a u$-equation in \eqref{sys-VF} leads to $\De Z^a p=\nab \cdot f_a$, which gives
$\|\nab Z^a p\|_{L^2}\les \|f_a\|_{L^2}$.
Thus we have
\begin{align*}
	\|N_k\|_{L^2}\les \sum_{|a|\leq k-1}\Big(t\|f_a\|_{L^2}+t\|g_a\|_{L^2}+\|(t+r)\mathcal N_a\|_{L^2}\Big)\,.
\end{align*}
According to the definitions of $f_a$ \eqref{fa}, $g_a$ \eqref{ga} and $\mathcal N_a$ \eqref{Na}, the above is bounded by
\begin{align*}
    \|N_k\|_{L^2}&\les \sum_{b+c=a;|a|\leq k-1}\<t\>\big\|(|Z^b u|+|Z^b G|+|Z^b \psi|)(|\nab Z^c u|+|\nab Z^c G|+|\nab Z^c\psi|)\big\|_{L^2}\\
 &\quad +\sum_{b+c=a;|a|\leq k-1} \big\| r |Z^b G \nab Z^c G|\big\|_{L^2}
 := I_1+I_2\,.
\end{align*}

By $N\geq 9$, the decays \eqref{Decay-phi2} and \eqref{Dec_psi}, we bound the first term $I_1$ by
\begin{align*}
    I_1\les E_{k}^{1/2} \<t\>\| |Z^{[(k-1)/2]+1}u|+|Z^{[(k-1)/2]+1}G|+|Z^{[(k-1)/2]+1}\psi|\|_{\lf}
    \les \ep_1 E^{1/2}_k\,.
\end{align*}
We use \eqref{AwayCone} to bound the second term,
\begin{align*}
    I_2\les E_k^{1/2}\|r\ |Z^{[(k-1)/2]+1}G|\|_{\lf} \les E_k^{1/2}E^{1/2}_{[(k-1)/2]+3}\les \ep_1 E_k^{1/2}\,.
\end{align*}
This completes the proof of \eqref{Nk-est}.
\end{proof}

Utilize the above three lemmas, we return to the proof of estimate \eqref{L2uH}.
\begin{proof}[Proof of the estimate \eqref{L2uH}]
	Noticing the definition of $X_k$ \eqref{XX}, applying $\<t-r\>\leq 1+|t-r|$ and \eqref{H-(t-r)}, we have
	\begin{align}       \nonumber
		X_{k}&=\sum_{|a|\leq k-1} \Big(\|\<t-r\>\nab Z^a u\|_{L^2}^2+\|\<t-r\>\nab Z^a G\|_{L^2}^2\Big)\\\nonumber
		&\leq 2 E_{k}+\sum_{|a|\leq k-1} \Big(\||t-r|\nab Z^a u\|_{L^2}^2+\||t-r|\nab Z^a G\|_{L^2}^2\Big)\\  \label{XuH-est}
		&\leq 6 E_{k}+C\sum_{|a|\leq k-1} \Big(\||t-r|\nab Z^a u\|_{L^2}^2+\||t-r|\nab\cdot Z^a G\|_{L^2}^2\Big)+C\sum_{|a|\leq k-1}\QQ_a\,.
	\end{align}
Thanks to
\begin{align*}
	\nab Z^a u&=\frac{1}{2}(\nab Z^a u+\nab \cdot Z^a G\otimes \om)+\frac{1}{2}(\nab Z^a u-\nab \cdot Z^a G\otimes \om),\\
	\nab \cdot Z^a G&=\frac{1}{2}(\nab Z^a u+\nab \cdot Z^a G\otimes \om)\om-\frac{1}{2}(\nab Z^a u-\nab \cdot Z^a G\otimes \om)\om\,,
\end{align*}
we obtain
\begin{align*}
	|t-r|(|\nab Z^a u|+|\nab\cdot Z^a G|)&\les |t+r||\nab Z^a u+\nab \cdot Z^a G\otimes \om|
	+|t-r||\nab Z^a u-\nab \cdot Z^a G\otimes \om|\,.
\end{align*}
The above estimate together with \eqref{goodun} implies that, for any $|a|\leq k-1$,
\begin{align}   \label{second-term}
	\||t-r|\nab Z^a u\|_{L^2}^2+\||t-r|\nab\cdot Z^a G\|_{L^2}^2\les E_k+\|N_k\|_{L^2}^2\,.
\end{align}
Collecting the estimates \eqref{XuH-est} and \eqref{second-term},  applying \eqref{Nk-est} and \eqref{Qa-est}, we arrive at
\begin{equation*}   
\begin{aligned}
X_{k}&\les E_{k}+\|N_k\|_{L^2}^2+\sum_{|a|\leq k-1}\QQ_a\\
&\les  E_{k}+\ep_1^2 E_k+\ep_1X_k+\ep_1 X_k^{1/2}E^{1/2}_{k-1}\leq C_3 \ep_1 X_k+\frac{C_3}{2}(1+\ep_1^2+\ep_1)E_k\,.
\end{aligned}	
\end{equation*}
Therefore the desired bound in \eqref{L2uH} is obtained.
\end{proof}

\subsection{Weighted estimates of profile for Schr\"odinger flow}
Next, we prove the $L^2$ weighted bound \eqref{Main_Prop_result2} for $Y_N$ and $Y_{N-2}$.
\begin{prop} \label{Prop_Psia}
	With the hypothesis in Proposition \ref{Main_Prop}, for any $t\in[0,T]$ and the profile $\Psi^{(a)}=e^{it\De}Z^a\psi$, we have
\begin{align}\label{YN-1}
&Y^{1/2}_{N}\leq 2E_N^{1/2}+ C_4 \ep_1^2\<t\>^\de\,,\\ \label{YN-3}
&Y^{1/2}_{N-2}\leq 2E_{N-2}^{1/2}+C_4\ep_1^2\<t\>^{-1/3+\de}\,.
\end{align}
\end{prop}
\begin{proof}
By scaling vector field $S=2t\d_t+x\cdot\nab$, the fact that $[e^{it\De},S]=0$ and \eqref{Main_Prop_result1}, we have
	\begin{equation}\label{keyIneq}
	\begin{aligned}
	&\lV x\cdot\nab|\nab|^{-1}\Psi^{(a)}\rV_{L^2}\leq  \lV S|\nab|^{-1}\Psi^{(a)}\rV_{L^2}+2\lV t\d_t|\nab|^{-1}\Psi^{(a)}\rV_{L^2}\\
	\leq &\  \lV S |\nab|^{-1} Z^a\psi\rV_{L^2}+2\lV t\d_t|\nab|^{-1}\Psi^{(a)}\rV_{L^2}
    \leq 2 E^{1/2}_{|a|+1}+2t\lV \d_t |\nab|^{-1}\Psi^{(a)}\rV_{L^2}\,.
	\end{aligned}
	\end{equation}
    Then it suffices to estimate the last term $\d_t|\nab|^{-1}\Psi^{(a)}$ in the right-hand side of \eqref{keyIneq}. By $\Psi^{(a)}=e^{it\De} Z^a\psi$ and the $\psi$-equation in \eqref{sys-VF}, we have
    \begin{align*}
        \|\d_t|\nab|^{-1}\Psi^{(a)}\|_{L^2}&= \||\nab|^{-1}\d_t(e^{it\De}Z^a\psi)\|_{L^2}  =  \|-ie^{it\De} |\nab|^{-1}(i\d_t Z^a\psi-\De Z^a\psi)\|_{L^2}\\
        &\leq \||\nab|^{-1}(i\d_t Z^a\psi-\De Z^a\psi)\|_{L^2}=\||\nab|^{-1}h_a\|_{L^2}\,.
    \end{align*}
    From the expression \eqref{ha} of $h_a$, it suffices to bound the following nonlinearities in $L^2$
    \begin{equation}\label{dtPsi_Term2}
    \begin{aligned}	
	|\nab|^{-1}h_a\approx &\sum_{b+c=a} |\nab|^{-1}\big( Z^b(u+A)\cdot\nab Z^c\psi+\nab Z^b u_j\cdot Z^c\psi_j\big)+
	 |\nab|^{-1}\NN^{(a)}\,,
	\end{aligned}
    \end{equation}
where $\NN^{(a)}$ is defined as
\begin{align*}
    \NN^{(a)}:&=\sum_{b+c=a}C_a^b Z^b A_{d+1} Z^c\psi+\sum_{b+c+e=a} C_a^{b,c} \big(Z^b (A+u)\cdot Z^c A Z^e\psi+i\Im(Z^b \psi_l Z^c\bar{\psi})Z^e\psi_l\big).
\end{align*}

\emph{Step 1. We prove the bound \eqref{YN-1}.}

\emph{i) Estimate of $Z^b (u+A)\cdot\nab Z^c\psi$.} By $\div Z^b u=0$, we move the gradient $\nab$ outside the formula. Then using H\"older's inequality, \eqref{Dec_psi} and \eqref{Decay-phi2}, we have
\begin{align*}
&\sum_{b+c=a}\lV |\nab|^{-1}( Z^bu\cdot\nab Z^c\psi)\rV_{L^2}
\les \sum_{b+c=a} \lV |\nab|^{-1}\nab\cdot(Z^b u Z^c\psi)\rV_{L^2}\\
&\les  \lV Z^{|a|} u \|_{L^3} \|Z^{[|a|/2]}\psi\rV_{L^6}+ \lV Z^{[|a|/2]} u \|_{\lf} \|Z^{|a|}\psi\rV_{L^2}\\
&\les E^{1/2}_{|a|+1}\<t\>^{-1}(Y^{1/2}_{[|a|/2]+1}+E^{1/2}_{[|a|/2]+1})+ \<t\>^{-1}(E^{1/2}_{[|a|/2]+2}+X^{1/2}_{[|a|/2]+2})E^{1/2}_{|a|}\,.
\end{align*}
For $|a|\leq N-1$, by \eqref{Main_Prop_Ass1} and \eqref{Main_Prop_Ass2}, we obtain
\begin{align}  \label{FirstTerm}
    \sum_{b+c=a;|a|\leq N-1}\lV |\nab|^{-1}( Z^bu\cdot\nab Z^c\psi)\rV_{L^2}
    \les \ep_1 \<t\>^{-1+\de}\ep_1+\<t\>^{-1+\de}\ep_1^2\les \ep_1^2\<t\>^{-1+\de}\,.
\end{align}
Using $\div Z^b A=0$, \eqref{Dec_d1/2-A} and \eqref{Dec_dA}, we can also bound the term $Z^b A\cdot\nab Z^c\psi$ by $\ep_1^3 \<t\>^{-3/2+\de}$ in a same way.

\emph{ii) Estimate of $\nab Z^b u\cdot Z^c\psi$.} 
By Leibniz formula, we have
\begin{align*}
    |\nab|^{-1}(\nab Z^b u_j\cdot Z^c\psi_j)
    = |\nab|^{-1}\nab( Z^b u_j\cdot Z^c\psi_j)-|\nab|^{-1}( Z^b u_j\cdot \d Z^c\psi_j)\,.
\end{align*}
Using the constraint \eqref{comm} and $\div Z^b u=0$, the second term in the right hand side can be rewritten as
\begin{align*}      \nonumber 
      Z^b u_j\cdot \d Z^c\psi_j
    &=Z^b u_j\cdot \d_j Z^c\psi-\sum_{c_1+c_2=c}i\big( Z^b u_j\cdot (Z^{c_1} A Z^{c_2}\psi_j-Z^{c_1} A_j Z^{c_2}\psi)\big)\\  
    &=\d_j( Z^b u_j\cdot  Z^c\psi)-\sum_{c_1+c_2=c}i\big( Z^b u_j\cdot (Z^{c_1} A Z^{c_2}\psi_j-Z^{c_1} A_j Z^{c_2}\psi)\big)\,.
\end{align*}
Hence, we obtain
\begin{align}    \label{ThirdTerm}
	\||\nab|^{-1}(\nab Z^b u_j\cdot Z^c \psi_j)\|_{L^2}\les \|Z^b uZ^c\psi\|_{L^2}+\sum_{c_1+c_2=c}\| |Z^b u|\ |Z^{c_1} A| |Z^{c_2}\psi|\|_{L^{6/5}}\,.
\end{align}

Now along the same line as the case \eqref{FirstTerm}, we bound the quadratic term $\|Z^b uZ^c\psi\|_{L^2}$ by $\ep_1^2\<t\>^{-1+\de}$.
For the cubic terms in the right hand side of \eqref{ThirdTerm}, by \eqref{Main_Prop_Ass1}, \eqref{Dec_d1/2-A}, \eqref{Decay-phi2} and \eqref{Dec_psi}, we have
\begin{equation}\label{Third-High}
    \begin{aligned}
    &\quad  \sum_{|b+c_1+c_2|\leq N-1}\| |Z^b u| | Z^{c_1} A| \ |Z^{c_2}\psi|\|_{L^{6/5}}\\
    &\les \|Z^N A\|_{L^2}(\|Z^{N-1}u\|_{L^2}\|Z^{[(N-1)/2]}\psi\|_{L^\infty}+\|Z^{[(N-1)/2]}u\|_{L^\infty}\|Z^{N-1}\psi\|_{L^2})\\
    &\les \ep_1^2\<t\>^{-1/2+\de}(\ep_1\<t\>^\de \ep_1\<t\>^{-1})\les \ep_1^4 \<t\>^{-3/2+2\de}\,.
\end{aligned}
\end{equation}
Hence, from the above two bounds, we obtain
\begin{align}   \label{Thi-Est}
    \sum_{|b+c|\leq N-1} \| |\nab|^{-1}(\nab Z^b u_j\cdot Z^c\psi_j)\|_{L^2}\les \ep_1^2\<t\>^{-1+\de}\,.
\end{align}

\emph{iii) Estimates of $\mathcal N^{(a)}$.} To bound $|\nab|^{-1}\mathcal N^{(a)}$ in $L^2$, from Sobolev embedding it suffices to bound $\NN^{(a)}$ in $L^{6/5}$. By \eqref{Dec_Ad+1-L2}, \eqref{Dec_psi} and \eqref{Dec_Ad+1L3} we have
\begin{align*}
    \sum_{|b+c|\leq N-1} \|Z^b A_{d+1} Z^c\psi\|_{L^{6/5}}
    &\les \|Z^{N-1} A_{d+1} \|_{L^2} \| Z^{[(N-1)/2]}\psi\|_{L^{3}}+\|Z^{[(N-1)/2]} A_{d+1} \|_{L^3} \| Z^{N-1}\psi\|_{L^{2}}\\
    &\les \ep_1^2\<t\>^{-1+\de}\ep_1\<t\>^{-1/2}+ \ep_1^2\<t\>^{-2+\de}\ep_1 \<t\>^\de
    \les \ep_1^3\<t\>^{-3/2+\de}\,.
\end{align*}
Since the decay estimates of $\|Z^{N-1} A\|_{L^2}$ and $\|Z^{[(N-1)/2]}A\|_{\lf}$ are better than that of $u$, thus by \eqref{Third-High} we also bound the $L^2$-norm of $Z^b (A+u)\cdot Z^c A Z^e\psi$ by $\ep_1^4 \<t\>^{-2+\de}$.
Similarly, by \eqref{Dec_psi} we can bound
\begin{align*}
    \sum_{|b+c+e|\leq N-1}\| \Im(Z^b \psi_l Z^c\bar{\psi})Z^e\psi_l\|_{L^{6/5}}\les \|Z^{N-1}\psi\|_{L^2}\|Z^{[(N-1)/2]}\psi\|_{L^6}^2\les \ep_1^3\<t\>^{-2+\de}\,.
\end{align*}
Hence, we obtain the estimate
\begin{align} \label{Forth-Est}
    \sum_{|a|\leq N-1}\| |\nab|^{-1}\mathcal N^{(a)}\|_{L^2}\lesssim \sum_{|a|\leq N-1}\| \mathcal N^{(a)}\|_{L^{6/5}}\les \ep_1^3\<t\>^{-3/2+\de}\,.
\end{align}

Collecting the estimates \eqref{FirstTerm}, \eqref{Thi-Est} and \eqref{Forth-Est}, the $\||\nab|^{-1}h_a\|_{L^2}$ for $|a|\leq N-1$ is bounded by $\ep_1^2\<t\>^{-1+\de}$,
which together with \eqref{keyIneq} gives the estimate \eqref{YN-1}.

 \smallskip 
\emph{Step 2. We prove the bound \eqref{YN-3}.}
	
\emph{i) Estimate of $Z^b u\cdot\nab Z^c\psi$ and $Z^b A\cdot \nab Z^c\psi$.} By $\div Z^b u=0$, we move the gradient $\nab$ outside the formula. Then using H\"older's inequality, \eqref{Decay-phi2} and \eqref{Dec_psi}, we have
\begin{equation} \label{First-N-3}
\begin{aligned}
&\sum_{|b+c|\leq N-3}\lV |\nab|^{-1}( Z^bu\cdot\nab Z^c\psi)\rV_{L^2}
\les \sum_{|b+c|\leq N-3} \lV |\nab|^{-1}\nab\cdot(Z^b u Z^c\psi)\rV_{L^2}\\
&\les  \lV Z^{N-3} u \|_{L^3} \|Z^{N-3}\psi\rV_{L^6}
\les \ep_1\<t\>^{-1/3+\de/3} \ep_1\<t\>^{-1}
\les \ep_1^2 \<t\>^{-4/3+\de}\,.
\end{aligned}
\end{equation}
In a same way, by $\div Z^b A=0$, \eqref{Dec_d1/2-A} and \eqref{Dec_psi}, we can bound the other term $Z^b A\cdot \nab Z^c\psi$ by $\ep_1^2 \<t\>^{-3/2+\de}$.

\emph{ii) Estimate of $\nab Z^b u_j\cdot Z^c\psi_j$.} From the formula \eqref{ThirdTerm}, by \eqref{First-N-3} and \eqref{Third-High}, we obtain
\begin{align*}
    &\sum_{|b+c|\leq N-3}\||\nab|^{-1}(\nab Z^b u_j\cdot Z^c\psi_j)\|_{L^2}\\
    &\les  \sum_{|b+c|\leq N-3}\|Z^b u\cdot Z^c\psi\|_{L^2}
    +\sum_{|b+c_1+c_2|\leq N-3}\||\nab|^{-1}(|Z^b u| |Z^{c_1} A| |Z^{c_2}\psi)|\|_{L^2}\\
    &\les \ep_1 \<t\>^{-4/3+\de}+\ep_1^4\<t\>^{-3/2+2\de}\les \ep_1 \<t\>^{-4/3+\de}\,.
\end{align*}

\emph{iii) Estimate of $\mathcal N^{(a)}$.} Using the similar argument to \eqref{Forth-Est}, we can bound the $L^2$-norm of the last term in \eqref{dtPsi_Term2} by $\ep_1^3\<t\>^{-3/2+\de}$.

From the above estimates, we obtain
\begin{equation}  \label{dtPsia}
    \sum_{|a|\leq N-3} \|\d_t |\nab|^{-1}\Psi^{(a)}\|_{L^2}\leq \sum_{|a|\leq N-3} \| |\nab|^{-1}h_a\|_{L^2}\les \ep_1^2\<t\>^{-4/3+\de}\,,
\end{equation}
which, combined with \eqref{keyIneq}, gives the bound \eqref{YN-3}.	This completes the proof of the lemma.
\end{proof}

\medskip 
\section{Proofs of the main theorems}\label{sec-proof}
With the above energy estimates in Proposition \ref{Prop_va} and weighted estimates in Proposition \ref{XX-prop} and \ref{Prop_Psia} at hand, we now turn our attention to the main theorems.  First, we prove the bootstrap Proposition \ref{Main_Prop}, then we obtain the global solution of \eqref{sys-2}-\eqref{constraints-re} with small data, and establish the scatterings \eqref{Scatter-Gre} and \eqref{scattering}. Second, we prove a lemma in order to bound the Schr\"odinger map $\phi$. This lemma together with Theorem \ref{Ori_thm} enable us obtain the main Theorem \ref{Ori_thm0}.

\subsection{Proof of Proposition \ref{Main_Prop} and Theorem \ref{Ori_thm}}

\begin{proof}[Proof of Proposition \ref{Main_Prop}]
\emph{i) Improved bounds \eqref{Main_Prop_result1}.}

From Proposition \ref{Prop_va}, \eqref{Main-ini}, $\ep_1=C_2C_1\ep_0$ and $\sqrt{\ep_0}\leq \de$, we obtain
\begin{align*}
    E_N(t)&\leq E_N(0)+ \int_0^t C_E\ep_1^3 \<s\>^{-1+2\de}\ ds\leq (\ep_1/C_2C_1)^2+C_E \ep_1^3 \de^{-1} \<t\>^{2\de}\leq \ep_1^2\<t\>^{2\de}/4,
\end{align*}
and
\begin{align*}
    E_{N-2}(t)&\leq E_{N-2}(0)+\int_0^t C_E \ep_1^3\<s\>^{-4/3+2\de}\ ds\leq \ep_0^2+5C_E\ep_1^3\leq  \ep_1^2/4\,,
\end{align*}
where $C_1,\ C_E\geq 1$ are fixed constants given in Lemma \ref{psi-by_phi-Lemma} and Proposition \ref{Prop_va}, $C_2\geq 3$ is an universal constant, and the $\ep_0>0$ is chosen to be $\ep_0\leq \min\{(8C_EC_2^3C_1^3)^{-2},\ (40C_EC_2^3C_1^3)^{-1}\}$.
This completes the proof of \eqref{Main_Prop_result1}.

\emph{ii) Improved bounds \eqref{Main_Prop_result2}.}

By Proposition \ref{XX-prop} and \eqref{Main_Prop_Ass2}, for the case $k=N-2\geq 7$ we have
\begin{align*}
	X_{N-2}\leq  C_3 \ep_1 C_w^2\ep_1^2+C_3\ep_1^2\leq \big(C_3 C_w^2 \ep_1+C_3\big)\ep_1^2 \leq (C_w\ep_1/4)^2\,,
\end{align*}
and for the case $k=N$, we deduce that
\begin{align*}
	X_{N} \leq C_3 \ep_1 C_w^2\ep_1^2\<t\>^{2\de}+C_3\ep_1^2\<t\>^{2\de}\leq \big(C_3 C_w^2\ep_1 +C_3\big)\ep_1^2\<t\>^{2\de}\leq (C_w\ep_1/4)^2 \<t\>^{2\de}\,,
\end{align*}
where we choose $C_w$ and $\ep_1$ to be $C_w\geq 6\sqrt{C_3},\  C_3\ep_1<\frac{1}{32}$.
From Proposition \ref{Prop_Psia} and \eqref{Main_Prop_Ass1}, we easily obtain that
$Y^{1/2}_{N}\leq (2 +C_4\ep_1)\ep_1\<t\>^\de\leq C_w \ep_1\<t\>^\de/4$ and $Y^{1/2}_{N-2}\leq (2+C_4\ep_1)\ep_1\leq C_w \ep_1/4$,
where we should choose $C_w$ to be $C_w\geq 8+4C_4$.

To conclude, from the above process, we choose
$C_2\geq 3$, $C_w\geq \max\{6\sqrt{C_3},8+4C_4\}$, and
$\ep_0\leq \min\{(8C_EC_2^3C_1^3)^{-2},(40C_EC_2^3C_1^3)^{-1}, (32 C_3C_2C_1)^{-1}\}$,
the improved bounds \eqref{Main_Prop_result1} and \eqref{Main_Prop_result2} holds. We completes the proof of Proposition \ref{Main_Prop}.
\end{proof}

Next, we use the continuity method to prove the Theorem \ref{Ori_thm}.
\begin{proof}[Proof of Theorem \ref{Ori_thm}]
By Proposition \ref{Main_Prop} and continuity method, the global solution of system \eqref{sys-2}-\eqref{A,Ad+1} is easily obtained.
Here we focus on the scattering \eqref{Scatter-Gre}.

\emph{i) We prove the scattering \eqref{Scatter-Gre}.}

First, we rewrite the elastic equations as a hyperbolic system. Applying Leray projection $\P=I+\nab(-\De)^{-1}\nab$ and $|\nab|^{-1}\d_j $ to the first equation in \eqref{sys-VF}. Using the constraints $\nab\cdot Z^a G^T=0$ and \eqref{curl-vf}, we rewrite the term with respect to $\nab\cdot Z^a G$ as
\begin{align*}
    &-|\nab|^{-1}\d_j \P \d_k Z^a G_{ik}=-|\nab|^{-1}\d_j  \d_k Z^a G_{ik}\\
    =&\  -|\nab|^{-1} \d_k \d_k Z^a G_{ij}-|\nab|^{-1} \d_k\NN_{a,jik}
    =|\nab| Z^aG_{ij}-|\nab|^{-1} \d_k\NN_{a,jik}\,.
\end{align*}
We obtain
\begin{equation*}
	\left\{
	\begin{aligned}
		&\d_t |\nab|^{-1}\d_j Z^a u_i+|\nab| Z^a G_{ij}= |\nab|^{-1}\d_j\P f_{a,i}+|\nab|^{-1}\d_k N_{a,jik}\,,\\
		&\d_t Z^a G-\nab Z^a u =  g_a\,,
	\end{aligned}
	\right.
\end{equation*}
We complexify the unknowns $u$ and $G$ as
$\mathbf{G}^{(a)}_{ij}:= Z^a G_{ij}+i|\nab|^{-1}\d_j Z^a u_i$, then
\begin{align*}
    (\d_t +i|\nab|)\mathbf{G}^{(a)}_{ij}=g_{a,ij}+i|\nab|^{-1}\d_j\P f_{a,i}+i|\nab|^{-1}\d_k N_{a,jik}\,.
\end{align*}

Denote the profile of $\mathbf{G}^{(a)}$ as
$\GG^{(a)}_{ij}:=e^{it|\nab|}\mathbf{G}^{(a)}_{ij}$,
we get
\begin{align*}
    \GG^{(a)}_{ij}(t)=\GG^{(a)}_{ij}(0)+ \int_0^t e^{is|\nab|}\Big( g_{a,ij}+i|\nab|^{-1}\d_j\P f_{a,i}+i|\nab|^{-1}\d_k N_{a,jik}\Big) ds\,.
\end{align*}
We claim that
\begin{align}   \label{claim}
    \| f_a\|_{L^2}+\| g_a\|_{L^2}+\|\NN_a\|_{L^2}\les \ep_1^2\<t\>^{-3/2+2\de}\,.
\end{align}

Assume that the claim \eqref{claim} holds, then the flow $\GG^{(a)}(t)$ is convergent in $L^2$. Indeed, for any $t_1\leq t_2$, we have
\begin{align*}
    \|\GG^{(a)}(t_1)-\GG^{(a)}(t_2)\|_{L^2}
    \les \int_{t_1}^{t_2} \| f_a\|_{L^2}+\| g_a\|_{L^2}+\|\NN_a\|_{L^2}\ ds\les \ep_1^2 \<t_1\>^{-1/2+2\de}\rightarrow 0\,,\quad \text{as }t_1,t_2\rightarrow\infty\,.
\end{align*}
Then the flow $\GG^{(a)}(t)$ is convergent as $t$ goes to infinity, whose limit is denoted as $\GG^{(a)}_{\infty}$.
Note that the imaginary and real parts are also go to $0$ as $t\rightarrow\infty$, then the scattering \eqref{Scatter-Gre} is obtained.

Finally, we turn our attention to the proof of claim \eqref{claim}. From the expressions of $f_a,\ g_a$ and $\mathcal N_a$ in \eqref{fa}, \eqref{ga} and \eqref{Na}, it suffices to bound
\begin{align*}
	\sum_{|b+c|\leq N-2} \big(\|Z^b u\cdot\nab Z^c (u,G)\|_{L^2}+\|\d_j  (Z^b u_i+Z^b G_{ik}) Z^c G_{jk}\|_{L^2}+\|Z^b\psi \nab Z^c\psi\|_{L^2}\big)\,.
\end{align*}
The first two terms have been considered in \eqref{k1-1} and \eqref{k1-2}, which are bounded by $\ep_1^2\<t\>^{-3/2+2\de}$. 
Therefore, the claim \eqref{claim} is obtained. We complete the proof of the scattering \eqref{Scatter-Gre}.

\smallskip 
\emph{ii) We prove the scattering \eqref{scattering}.}

By the definition of profile $\Psi^{(a)}=e^{it\De} Z^a\psi$
we have
    \begin{align*}
        \Psi^{(a)}(t)=\Psi^{(a)}(0)+\int_0^t \d_s\Psi^{(a)}(s)ds
        =\Psi^{(a)}(0)-i\int_0^t e^{is\De} h_a(s) ds\,.
    \end{align*}
    We see the derivative $\nab$ as a vector field, then by \eqref{dtPsia} we obtain for any $|a|\leq N-4$
    \begin{align*}
        \sum_{|a|\leq N-4}\|\Psi^{(a)}(t_1)-\Psi^{(a)}(t_2)\|_{H^{-1}}\lesssim &\  \int_{t_1}^{t_2} \sum_{|a|\leq N-3}\|\d_s|\nab|^{-1}\Psi^{(a)}(s)\|_{L^2}ds\\
        \lesssim &\  \int_{t_1}^{t_2}\ep_1^2\<s\>^{-4/3+\de}ds\rightarrow 0,\quad {\rm as}\ t_1,t_2\rightarrow\infty.
    \end{align*}
This implies that the $\Psi^{(a)}(t)$ is convergent in $L^2$ as $t\rightarrow\infty$, and motivates us to define the function $\Psi^{(a)}_{\infty}$ as
    \begin{equation*}   
    	\Psi^{(a)}_{\infty}:=\Psi^{(a)}(0)-i\int_0^\infty e^{is\De} h_a(s) ds.
    \end{equation*}
    Then we have for any $|a|\leq N-4$,
\begin{equation} \label{sca-psi}
\begin{aligned}
        &\ \|Z^a\psi-e^{-it\De}\Psi^{(a)}_\infty\|_{H^{-1}}=\|e^{it\De}Z^a\psi-\Psi^{(a)}_\infty\|_{H^{-1}}\\
        \leq &\ \int_t^\infty \sum_{|a|\leq N-3}\|\d_s|\nab|^{-1}\Psi^{(a)}(s)\|_{L^2} ds
        \lesssim \ep_1^2 \<t\>^{-1/3+\de}\rightarrow 0\,,\quad {\rm as}\ t\rightarrow \infty\,.
    \end{aligned}
\end{equation}
    This shows that $Z^a\psi$ for $|a|\leq N-4$ scatters to the linear solution $e^{it\De}\Psi^{(a)}_\infty$ in the Sobolev space $H^{-1}$. We complete the proof of \eqref{scattering}.
\end{proof}

\subsection{Proof of Theorem \ref{Ori_thm0}}

In order to construct the Schr\"odinger flow $\phi$ from the solution $\psi$ of \eqref{sys-2}, we need the following lemma.
\begin{lemma}[Bounds for $\phi,\ v$ and $w$]  \label{phi-bypsi}
     Assume that $\psi$ is a solution of \eqref{sys-2}-\eqref{A,Ad+1}. If the differentiated fields $\psi$ has the additional property for any fixed time
\begin{align} \label{Ass-psi1}
&\|Z^N\psi(t)\|_{L^2}+\||\nab|^{-1}Z^N\psi(t)\|_{L^{2}}\les \ep_0\<t\>^\de\,,\\\label{Ass-psi2}
&\|Z^{N-2}\psi(t)\|_{L^2}+\||\nab|^{-1}Z^{N-2}\psi(t)\|_{L^{2}}\les \ep_0\,.
\end{align}
Then we have the bounds for the map $\phi$
\begin{align}    \label{Result-psi1}
&\|\nab Z^N\phi(t)\|_{L^2}\leq C_5\ep_0\<t\>^\de\,,\qquad \|\nab Z^{N-2}\phi(t)\|_{L^2}\leq C_5\ep_0\,,\\\label{Result-psi2}
&\|Z^N(\phi(t)-Q)\|_{L^{2}}\leq C_5\ep_0\<t\>^\de\,,\qquad \|Z^{N-2}(\phi(t)-Q)\|_{L^{2}}\leq C_5 \ep_0\,,  
\end{align}
Moreover, we have the decay estimates for Schr\"odinger map $\phi$ and the frame $(v,\ w)$
\begin{align}  \label{Dec_dphi}
    \|\nab Z^{N-1} \phi\|_{L^3}+\|\nab Z^{N-1} v\|_{L^3}+\|\nab Z^{N-1} w\|_{L^3}\les \ep_0 \<t\>^{-1/2+\de}\,.
\end{align}
\end{lemma}
\begin{proof}
     \emph{i) We prove the bounds \eqref{Result-psi1}.}

From the formula \eqref{Ell-A}, Sobolev embeddings and \eqref{Ass-psi2}, for $j=N-2$ or $N$ we have
\begin{align}   \label{ZA}
	\|Z^j A\|_{L^2}\les \||\nab|^{-1}(Z^j\psi Z^{[j/2]}\psi)\|_{L^2}\les \|Z^j\psi\|_{L^2} \|Z^{[j/2]}\psi\|_{L^3}\les \ep_0\|Z^j\psi\|_{L^2}\,.
\end{align}
Applying the vector fields $Z$ to the formulas \eqref{phi-v-w} yields
     \begin{equation}   \label{phi-v-wVF}
     	\left\{ \begin{aligned}
     		&\d_m Z^a\phi=\sum_{b+c=a}(Z^bv\Re Z^c\psi_m+Z^bw\Im Z^c\psi_m)\,,\\
     		&\d_m Z^av=\sum_{b+c=a}(-Z^b\phi\Re Z^c\psi_m+Z^bw Z^cA_m)\,,\\
     		&\d_m Z^aw=\sum_{b+c=a}(-Z^b\phi\Im Z^c\psi_m- Z^b v Z^cA_m)\,.
     	\end{aligned}
     	\right.
     \end{equation}
Then by \eqref{phi-v-wVF}, \eqref{ZA}, \eqref{Ass-psi2} and Sobolev embedding we arrive at
 \begin{equation}\label{dZphi}
 	\begin{aligned}   
 		\|\nab Z^j(\phi,v,w) \|_{L^2}&\lesssim \| Z^j(\psi,A) \|_{L^2}\|(Z^{[j/2]}\phi, Z^{[j/2]} v, Z^{[j/2]} w)\|_{L^\infty}\\
 		&\quad +\|(Z^{[j/2]}\psi,Z^{[j/2]}A)\|_{L^3}\sum_{[j/2]<|c|\leq j}\|(Z^c\phi,Z^cv,Z^cw)\|_{L^6}\\
 		&\lesssim  (1+\ep_0)\|Z^j\psi\|_{L^2}(1+\|\nab Z^{[j/2]+1}(\phi,v,w) \|_{L^2})\\
 		&\quad +\ep_0 \|\nab Z^{j}(\phi,v,w) \|_{L^2}\,.
 	\end{aligned}
 \end{equation}
When $j=N-2$, from \eqref{Ass-psi2}, the estimate \eqref{dZphi} implies
\begin{align*}
    \|\nab Z^{N-2}(\phi,v,w) \|_{L^2}\les (1+\ep_0)\ep_0 (1+\|\nab Z^{N-2}(\phi,v,w) \|_{L^2})+\ep_0 \|\nab Z^{N-2}(\phi,v,w) \|_{L^2}\,,
\end{align*}
which also gives 
\begin{align} \label{phivw-Bd}
    \|\nab Z^{N-2}(\phi,v,w) \|_{L^2}\les \ep_0\,.
\end{align}
When $j=N$, from \eqref{Ass-psi1} and \eqref{phivw-Bd}, the estimate \eqref{dZphi} gives
\begin{align*}
    \|\nab Z^{N}(\phi,v,w) \|_{L^2}\les (1+\ep_0)\ep_0\<t\>^\de (1+\ep_0)+\ep_0\|\nab Z^{N}(\phi,v,w) \|_{L^2}\,,
\end{align*}
which also yields
\begin{align} \label{phivw-BdHi}
    \|\nab Z^{N}(\phi,v,w) \|_{L^2}\les \ep_0\<t\>^\de\,.
\end{align}

\emph{ii) We prove the bound \eqref{Result-psi2}.}

The $Z^a(\phi-Q)$ can be expressed using the first formula in \eqref{phi-v-wVF}. Since $Z^b v$ and $Z^b w$ are estimated similarly, here we only consider the first term $Z^b v\Re Z^c\psi_m$. By Leibniz rule, it suffices to bound
\begin{align*}
    &\ \sum_{|b+c|= |a|}\|\De^{-1}\d(Z^b v  Z^c\psi)\|_{L^2}
    \les \sum_{|b+c|= |a|}\|\De^{-1}\nab\d(Z^b v |\nab|^{-1} Z^c\psi+\De^{-1}\d(\nab Z^b v |\nab|^{-1} Z^c\psi)\|_{L^2}\\
    \les&\  \sum_{|b+c|= |a|}\||Z^b v |\nab|^{-1}Z^c\psi||\|_{L^2}
     +\sum_{|b+c|= |a|}\||\nab Z^b v |\nab|^{-1}Z^c\psi|\|_{L^{6/5}}\,.
\end{align*}
By $|v|=|w|=1$, \eqref{phivw-Bd} and \eqref{Ass-psi2}, we get
\begin{align*}
    \sum_{|b+c|\leq j}\||Z^b v |\nab|^{-1}Z^c\psi|\|_{L^2}
    &\les \||Z^{[j/2]}v||\|_{\lf} \||\nab|^{-1}Z^j\psi\|_{L^2}+\sum_{[j/2]<|b|\leq j}\||Z^b v|\|_{L^6} \||\nab|^{-1}Z^{[j/2]}\psi\|_{L^3}\\
    &\les (1+\ep_0)\||\nab|^{-1}Z^j\psi\|_{L^2}+\|\nab Z^{j}v \|_{L^2}\ep_0\,.
\end{align*}
By Sobolev embedding, \eqref{phivw-Bd} and \eqref{Ass-psi2}, we bound the second term by
\begin{align*}
    \sum_{|b+c|\leq j}\||\nab Z^b v |\nab|^{-1}Z^c\psi|\|_{L^{6/5}}
    \les&\ \||\nab Z^{[j/2]}v|\|_{L^3} \||\nab|^{-1}Z^j\psi\|_{L^2}
    +\||\nab Z^j v|\|_{L^2} \||\nab|^{-1}Z^{[j/2]}\psi\|_{L^3}\\
    \les&\  \ep_0\||\nab|^{-1}Z^j\psi\|_{L^2}+\||\nab Z^j v|\|_{L^2}\ep_0.
\end{align*}
From the above three estimates, for $|a|\leq N-2$, from \eqref{Ass-psi2} and \eqref{phivw-Bd} we have
\begin{align*}
    &\|Z^{N-2}(\phi-Q)\|_{L^2}\les (1+\ep_0)\||\nab|^{-1}Z^{N-2}\psi\|_{L^2}+\|\nab Z^{N-2}(v,w) \|_{L^2}\ep_0\les (1+\ep_0)\ep_0+\ep_0^2\les \ep_0\,,
\end{align*}
and for $|a|\leq N$, from \eqref{Ass-psi1} and \eqref{phivw-BdHi} we obtain
\begin{align*}
    &\|Z^{N}(\phi-Q)\|_{L^2}\les (1+\ep_0)\||\nab|^{-1}Z^{N}\psi\|_{L^2}+\|\nab Z^{N}(v,w) \|_{L^2}\ep_0\les (1+\ep_0)\ep_0\<t\>^\de+\ep_0\<t\>^\de \ep_0\les \ep_0\<t\>^\de\,.
\end{align*}
Therefore, the estimates \eqref{Result-psi2} are obtained.

\emph{iii) We prove the bound \eqref{Dec_dphi}.}

By the first formula in \eqref{phi-v-wVF}, we have
\begin{align*}
    \|\nab Z^{N-1}\phi\|_{L^3}&\les \sum_{|b+c|\leq N-1} \|(|Z^b v|+|Z^b w|)Z^c\psi\|_{L^3}\\
    &\les \||Z^{[(N-1)/2]} v|+|Z^{[(N-1)/2]} w|\|_{\lf}\|Z^{N-1}\psi\|_{L^3}\\
    &\quad + \sum_{[(N-1)/2]<|b|\leq N-1}\||Z^{b} v|+|Z^{b} w|\|_{L^6}\|Z^{[(N-1)/2]}\psi\|_{L^6}\\
    &\les \big(1+\||\nab Z^{[(N-1)/2]+1} v|+|\nab Z^{[(N-1)/2]+1} w|\|_{L^2}\big)\|Z^{N-1}\psi\|_{L^2}^{1/2}\|Z^{N-1}\psi\|_{L^6}^{1/2}\\
    &\quad + \||\nab Z^{N-1} v|+|\nab Z^{N-1} w|\|_{L^2}\|Z^{[(N-1)/2]}\psi\|_{L^6}\,.
\end{align*}
Using \eqref{phivw-Bd}, \eqref{Dec_psi}, \eqref{phivw-BdHi} and \eqref{Dec_psi}, we can further bound the above terms by
\begin{align*}
    \|\nab Z^{N-1}\phi\|_{L^3}&\les (1+\ep_0 )\ep_0 \<t\>^{\de/2} \<t\>^{-1/2+\de/2}+\ep_0\<t\>^\de \ep_0 \<t\>^{-1} 
    \les \ep_0 \<t\>^{-1/2+\de}\,.
\end{align*}

From the second formula in \eqref{phi-v-wVF}, for $ |a|\leq N-1$, by Sobolev embedding and interpolation we have
\begin{align*}
	\|\nab Z^a v\|_{L^3}\les& \sum_{b+c=a}\|Z^b\phi Z^c\psi+Z^b w Z^c A\|_{L^3}\\
	\les &\ \|Z^{[|a|/2]}\phi\|_{\lf}\|Z^{|a|}\psi\|_{L^3}+\sum_{[|a|/2]<|b|\leq |a|}\|Z^b\phi\|_{L^6}\|Z^{[|a|/2]}\psi\|_{L^6}\\
	&\ +\|Z^{[|a|/2]}w\|_{\lf}\|Z^{|a|}A\|_{L^3}+\sum_{[|a|/2]<|b|\leq |a|}\|Z^b w\|_{L^6}\|Z^{[|a|/2]}A\|_{L^6}\\
	\les &\ (1+\|\nab Z^{[N/2]+1}(\phi,w)\|_{L^2})(\|Z^{N-1}\psi\|_{L^3}+\|Z^N A\|_{L^2})\\
	&\ +\|\nab Z^{N-1}(\phi,w)\|_{L^2} (\|Z^{N-3}\psi\|_{L^6}+\|Z^{N-3}A\|_{L^2}^{1/3}\|Z^{N-3}A\|_{\lf}^{2/3})\,.
\end{align*}
By \eqref{phivw-Bd}, \eqref{Dec_psi}, \eqref{Dec_d1/2-A}, \eqref{phivw-BdHi}, \eqref{Dec_psi} and \eqref{Dec_dA}, we bound the above by
\begin{align*}
	\|\nab Z^a v\|_{L^3}\les &\ (1+\ep_0)(\ep_1\<t\>^{-1/2+\de}+\ep_1^2\<t\>^{-1/2+\de})+\ep_0\<t\>^\de(\ep_1\<t\>^{-1}+\ep_1^2\<t\>^{-3/2+2\de})
	\les \ep_0\<t\>^{-1/2+\de}\,.
\end{align*} 
Similarly, we also bound the $\|\nab Z^a w\|_{L^3}$ with $|a|\leq N-1$ by $\ep_0\<t\>^{-1/2+\de}$. Thus the bound \eqref{Dec_dphi} follows. We complete the proof of the lemma. 
\end{proof}

Finally, we can use the Theorem \ref{Ori_thm} to prove our main Theorem \ref{Ori_thm0}.
\begin{proof}[Proof of Theorem \ref{Ori_thm0}]

\emph{i) Global existence of \eqref{ori_sys}.}

From the assumption \eqref{Main-ini0} and local well-posedness of \eqref{ori_sys} in Theorem \ref{main-Thm}, we obtain the unique local solution $(u,F,\phi)\in H^3\times H^3\times H^4_Q$ on some time interval $[0,T]$. 
We then extend the local solution by bootstrap argument and continuity method. Assume that there exists maximal time $T_0$ such that for any $0\leq t\leq T_0$
\begin{align}  \label{BTAss0}
&\sum_{|a|\leq N}\big(\|Z^a (u,F-I)\|_{L^2}+\|  Z^a (\phi-Q)\|_{L^2} +\|\nab Z^a \phi\|_{L^2}\big)\leq C_0\ep_0\<t\>^\de\,,\\ \label{BTAss1}
&\sum_{|a|\leq N-2}\big(\|Z^a (u,F-I)\|_{L^2}+\|  Z^a (\phi-Q)\|_{L^2} +\|\nab Z^a \phi\|_{L^2}\big)\leq C_0\ep_0\,.
\end{align}
Then under the Coulomb gauge, the system \eqref{ori_sys} on $[0,T_0]$ is rewritten as \eqref{sys-2}. We can also obtain the associated solution $(u,G,\psi)$ on $[0,T_0]$ from $(u,F,\phi)$ with initial data $(u_0,G_0,\psi_0)$ satisfying \eqref{Main-ini} by Lemma \ref{psi-by_phi-Lemma}. 
By Theorem \ref{Ori_thm}, the system \eqref{sys-2}-\eqref{constraints-re} admits a unique global solution satisfying
\begin{align*}
    E^{1/2}_N(u,G,\psi;t)\leq C_2C_1\ep_0\<t\>^\de\,,\qquad  E^{1/2}_{N-2}(u,G,\psi;t)\leq C_2C_1\ep_0\,.
\end{align*}
Using the Lemma \ref{phi-bypsi}, we can bound the energy of \eqref{ori_sys} again
\begin{align*}
    \sum_{|a|\leq N}\big(\|Z^a (u,F-I)\|_{L^2}+\|  Z^a (\phi-Q)\|_{L^2} +\|\nab Z^a \phi\|_{L^2}\big)\leq C_5\ep_0\<t\>^\de\leq \frac{C_0}{2}\ep_0\<t\>^\de,\\
    \sum_{|a|\leq N-2}\big(\|Z^a (u,F-I)\|_{L^2}+\|  Z^a (\phi-Q)\|_{L^2} +\|\nab Z^a \phi\|_{L^2}\big)\leq C_5\ep_0\leq \frac{C_0}{2}\ep_0,
\end{align*}
Where we choose $C_0$ to be $C_0\geq 2C_5$.
Then by local existence, the solution of \eqref{ori_sys} can be extended forward, and the energy bounds \eqref{BTAss0} and \eqref{BTAss1} also hold on the larger time interval $[0,T_0+\ep]$ for some $\ep>0$. This contracts to the assumption on $T_0$. Hence, the system \eqref{ori_sys} admits global solution with energy bounds \eqref{energybd0}.

\emph{ii) We prove the scattering \eqref{scattering0}.}

From the first formula in \eqref{phi-v-wVF}, we express the $Z^a(\phi-Q)$ for any $0\leq |a|\leq N$ as 
\begin{align}  \label{Form-phi}
	Z^a(\phi-Q)=\De^{-1}\d_m(\sum_{b+c=a}Z^b v \Re Z^c\psi_m+Z^b w \Im Z^c\psi_m)\,.
\end{align}
We claim that for any $|a|\leq N-4$, there holds
\begin{equation}  \label{Claim-scatt}
	\lim_{t\rightarrow\infty}\|Z^a(\phi-Q)-\Re e^{-it\De}(v_\infty \De^{-1}\d_m  \Psi^{(a)}_{\infty,m})-\Im e^{-it\De}(w_\infty \De^{-1}\d_m \Psi^{(a)}_{\infty,m})\|_{H^1}=0\,.
\end{equation}
Then the scattering \eqref{scattering0} follows, where the $\Phi_{1,a}$ and $\Phi_{2,a}$ are chosen to be
$\Phi_{1,a}:=v_\infty \De^{-1}\d_m  \Psi^{(a)}_{\infty,m}$, $\Phi_{2,a}:=w_\infty \De^{-1}\d_m \Psi^{(a)}_{\infty,m}$. In the remaining part, we focus on the proof of \eqref{Claim-scatt}.

\emph{a) We prove the claim \eqref{Claim-scatt} in energy space $\dot{H}^1$.}

By the formula \eqref{Form-phi}, we rewrite the left hand side of \eqref{Claim-scatt} as 
\begin{align*}
	&Z^a(\phi-Q)-\Re e^{-it\De}(v_\infty \De^{-1}\d_m  \Psi^{(a)}_{\infty,m})-\Im e^{-it\De}(w_\infty \De^{-1}\d_m \Psi^{(a)}_{\infty,m})\\
	\leq  &\ |\De^{-1}\d_m(\sum_{b+c=a}Z^b v \Re Z^c\psi_m)-\Re \De^{-1}\d_m(v_\infty  e^{-it\De} \Psi^{(a)}_{\infty,m})|\\
	&\ +|\De^{-1}\d_m(\sum_{b+c=a}Z^b w \Im Z^c\psi_m)-\Im \De^{-1}\d_m(w_\infty  e^{-it\De} \Psi^{(a)}_{\infty,m})|
	:= I_1+I_2\,.
\end{align*}
For the first term, by $v_{\infty}=\lim_{|x|\rightarrow \infty}v(t,x)$ we divide the $I_1$ into
\begin{align}   \label{S1}
	I_1\leq &\ \sum_{b+c=a;|b|\geq 1}\big|\De^{-1}\d_m(Z^b v Z^c\psi_m)\big|+\big|\De^{-1}\d_m( (v-v_\infty) Z^a\psi_m)\big|\\\nonumber
	&\ +\big|\De^{-1}\d_m(v_\infty (Z^a\psi_m- e^{-it\De} \Psi^{(a)}_{\infty,m}))\big|\,,
\end{align}
then by H\"older's inequality, in $\dot{H}^1$ we have
\begin{align*}
	\|I_1\|_{\dot{H}^1}\les &\ \sum_{b+c=a;|b|\geq 1}\|Z^b v  Z^c\psi\|_{L^2}+ \|(v-v_\infty) Z^a\psi\|_{L^2}+\|Z^a\psi_m-e^{-it\De} \Psi^{(a)}_{\infty,m}\|_{L^2}\\
	\les  &\ \sum_{b+c=a;|b|\geq 1} \|Z^b v\|_{L^6}\|Z^c\psi\|_{L^3}+ \|v-v_\infty\|_{\lf}\|Z^a\psi\|_{L^2}+\|Z^a\psi_m-e^{-it\De} \Psi^{(a)}_{\infty,m}\|_{L^2}\,.
\end{align*}
Using \eqref{Dec_dphi}, \eqref{Dec_psi} \eqref{v-cong} and \eqref{scattering}, for $|a|\leq N-2$, the $I_1$ can be controlled by
\begin{align*}
	I_1
	\les  \ep_1 \ep_1\<t\>^{-1/2}+ \ep_1\<t\>^{-1/2+\de}\ep_1+\ep_1^2\<t\>^{-1/3+\de}
	\les \ep_1^2\<t\>^{-1/3+\de}\longrightarrow 0\,,\qquad  \text{as }t\rightarrow \infty\,.
\end{align*}
In a same way, the second term $II_2$ also converges to $0$ as $t\rightarrow\infty$. Hence, the claim \eqref{Claim-scatt} in $\dot{H}^1$ is obtained.

\emph{b) We prove the claim \eqref{Claim-scatt} in mass space $L^2$.}

In a similar way as above $a)$, it suffices to bound $I_1$ and $I_2$ in $L^2$. For the first term in \eqref{S1}, by Leibniz rule, Sobolev embedding, \eqref{Dec_dphi} and \eqref{Dec_psi}, we have 
\begin{align*}
	&\ \sum_{b+c=a;|b|\geq 1}\|\De^{-1}\d_m(Z^b v  Z^c\psi_m)\|_{L^2}\\
	\les&\ \sum_{b+c=a;|b|\geq 1}\|\De^{-1}\d_m\nab(Z^b v  |\nab|^{-1}Z^c\psi)+\De^{-1}\d_m(\nab Z^b v  |\nab|^{-1}Z^c\psi)\|_{L^2}\\
	\les&\ \sum_{b+c=a;|b|\geq 1}\|Z^b v\|_{\lf}  \||\nab|^{-1}Z^c\psi\|_{L^2}+\|\nab Z^b v\|_{L^3}\||\nab|^{-1}Z^c\psi\|_{L^2}
	\les \ep_1^2\<t\>^{-1/2+2\de}\,.
\end{align*}
Using Leibniz rule to rewrite the second term in \eqref{S1}, then by \eqref{v-cong} we have
\begin{align*}
	&\ \|\De^{-1}\d_m(( v-v_\infty)  Z^a\psi_m)\|_{L^2}\\
	\les&\ \big\||\De^{-1}\d_m \nab\big(( v-v_\infty) |\nab|^{-1}Z^a\psi_m\big)|+|\De^{-1}\d_m\big(\nab( v-v_\infty) |\nab|^{-1}Z^a\psi_m\big)|\big\|_{L^2}\\
	\les &\ \big\|( v-v_\infty)  |\nab|^{-1}Z^a\psi\big\|_{L^2}+\big\|\nab ( v-v_\infty)  |\nab|^{-1}Z^a\psi\big\|_{L^{6/5}}\\
	\les &\ \| v-v_\infty\|_{\lf}\|  |\nab|^{-1}Z^a\psi\big\|_{L^2}+\|\nab ( v-v_\infty)\|_{L^3}\|  |\nab|^{-1}Z^a\psi\|_{L^2}\\
	\les &\ \ep_1^2\<t\>^{-1/2+\de}\,.
\end{align*}
Since $v_\infty$ is a fixed vector, by the estimate \eqref{sca-psi} we bound the last term in \eqref{S1} by
\begin{align*}
\|\De^{-1}\d_m(v_\infty (Z^a\psi_m- e^{-it\De} \Psi^{(a)}_{\infty,m}))\|_{L^2}
\les \||\nab|^{-1}Z^a\psi-e^{-it\De}|\nab|^{-1}\Psi^{(a)}_{\infty,m}\|_{L^2}\les \ep_1^2\<t\>^{-1/3+\de}.
\end{align*}
Then, the $L^2$-norm of $I_1$ is bounded by $\ep_1^2\<t\>^{-1/3+\de}$, and goes to $0$ as $t\rightarrow\infty$. We can also obtain a similar estimate for the term $I_2$ in $L^2$. Hence, the claim \eqref{Claim-scatt} in $L^2$ follows.

This completes the proof of scattering \eqref{scattering0}.
\end{proof}

\medskip
\section*{Acknowledgment}
X. Hao is supported by the CAEP Foundation (Grant No:CX20210020). J. Huang is supported by National Natural Science Foundation of China (Grant No. 12301293) and Beijing Institute of Technology Research Fund Program for Young Scholars. The author N. Jiang is supported by the grants from the National Natural Foundation of China under contract Nos. 11971360 and 11731008, and also supported by the Strategic Priority Research Program of Chinese Academy of Sciences, Grant No. XDA25010404. L. Zhao is supported by NSFC Grant of China No. 12271497,  No. 12341102 and the National Key Research and Development Program of China No. 2020YFA0713100.

\end{document}